\documentclass[11pt,reqno,UTF8,twoside]{amsart}
\usepackage{amssymb}
\usepackage[dvips]{epsfig}
\usepackage{booktabs}
\usepackage{multirow}
\usepackage{bbm}
\usepackage{mathrsfs}
\usepackage{xltabular}
\usepackage{tabularx}
\usepackage{makecell} 
\usepackage{ragged2e} 
\usepackage{graphicx}
\usepackage{color}
\usepackage{amsmath}
\allowdisplaybreaks
\usepackage{amssymb}
\usepackage{amsfonts}
\usepackage{geometry}
\usepackage{xstring}
\usepackage{amsthm}
\usepackage[normalem]{ulem}    
\usepackage{bm}
\usepackage{cite}
\usepackage{tikz}
\usetikzlibrary{arrows.meta}

\makeatletter
\@namedef{subjclassname@2020}{\textup{2020} Mathematics Subject 
	Classification}
\makeatother

\newcommand{\Z}{\mathbb{Z}}

\newcommand{\R}{\mathbb{R}}

\newcommand{\res}{{\rm Res}}

\newcommand{\p}{\phi}

\renewcommand{\b}{\beta}

\newcommand{\mcH}{\mathcal{H}}
\newcommand{\mcP}{\mathcal{P}}

\newcommand{\mcI}{\mathcal{I}}

\newcommand{\mcM}{\mathcal{M}}

\newcommand{\mcW}{\mathcal{W}}

\newcommand{\mcN}{\mathcal{N}}

\newcommand{\one}{\mathbf 1}

\newcommand{\dist}{\operatorname{dist}}

\newcommand{\dd}{\mathrm d}
\newcommand{\econst}{\mathrm e}

\newtheorem{thm}{Theorem}[section]
\newtheorem{lem}[thm]{Lemma}
\newtheorem{prop}[thm]{Proposition}
\newtheorem{cor}{\bf Corollary}[section]
\newtheorem{Def}{Definition}[section]

\newtheorem{rmk}{\bf Remark}[section]

\theoremstyle{definition}

\numberwithin{equation}{section}

\usepackage[colorlinks=true,pdfstartview=FitV,linkcolor=magenta,citecolor=cyan]{hyperref}

\usepackage{caption}

\begin{document}
\title[Discrete Quantitative Unique Continuation]{Quantitative Unique Continuation on Simplex and $\Z^n$}

\author[Liu]{Shihe Liu}
\address[S. Liu] {School of Mathematical Sciences,
Peking University,
Beijing 100871,
China}
\email{2301110021@stu.pku.edu.cn}

\author[Shi]{Yunfeng Shi}
\address[Y. Shi] {School of Mathematics,
Sichuan University,
Chengdu 610064,
China}
\email{yunfengshi@scu.edu.cn}

\author[Zhang]{Zhifei Zhang}
\address[Z. Zhang] {School of Mathematical Sciences,
Peking University,
Beijing 100871,
China}
\email{zfzhang@math.pku.edu.cn}

\begin{abstract}
	We establish the quantitative unique continuation on both $n$-dimensional simplex lattice $\Delta^{(n)}_{N}$ and $\Z^n$ for arbitrary $n\geq 3$.

\end{abstract}


\maketitle

\tableofcontents

\section{Introduction}\label{sec:introduction}

Unique continuation concerns how rapidly a nontrivial solution of a given equation can decay and on what sets it can vanish. In the continuum $\R^d$, this subject has been extensively studied. For example, under standard assumptions, a nontrivial solution of a continuum Schr\"odinger equation cannot vanish on a nonempty open set. We are particularly interested in quantitative unique continuation (QUC), which provides explicit bounds on the decay of solutions rather than merely ruling out their vanishing. Bourgain and Kenig~\cite[Lemma 3.10]{BK05} proved that nontrivial solutions of the stationary Schr\"odinger equation with bounded potential cannot decay faster than
\[
\exp\!\left(-O\!\left(|x|^{\frac{4}{3}}\log |x|\right)\right)
\]
near every point of $\R^n$. More recently, Logunov, Malinnikova, Nadirashvili, and Nazarov~\cite{LMNN25} established the decay rate
\[
\exp\!\left(-C|x| (\log |x|)^{\frac32}\right)
\]
for real-valued solutions of the Schr\"odinger equation with real potential on $\R^2$, thereby resolving the two-dimensional Landis conjecture in the affirmative. In dimensions $n\geq 3$, the Landis conjecture was disproved by Frank and Ivanisvili~\cite{frank2026counterexamples}.

The quantitative unique continuation problem on graphs, and in particular on the discrete lattice $\Z^n$, appears to be more complicated. We focus on the equation
\begin{equation}\label{eq:stationary-Schrodinger}
    \Delta_{\Z^n}u=Vu
    \quad \text{on } \Z^n,
    \qquad
    \|V\|_{\infty}<\infty.
\end{equation}

When $n=1$, QUC on $\Z$ is elementary because the cone property controls the propagation of solutions. However, if the free Laplacian in \eqref{eq:stationary-Schrodinger} is replaced by a long-range hopping operator, Liu, Shi, and Zhang~\cite{LSZ26b} showed that QUC does not hold for all hopping operators and that its validity is closely related to the essential singularities of the hopping symbol; see \cite[Counterexample 2.2]{LSZ26b}. When the hopping operator has a Laurent symbol, QUC is valid, as proved in \cite[Theorem 2.3]{LSZ26b}.

When $n=2$, a well-known example due to Jitomirskaya~\cite[Theorem 2]{Jit07} first showed that the continuum QUC result of \cite{BK05} cannot be extended directly to $\Z^2$, since the support of a solution may be concentrated on a low-dimensional subset of $\Z^2$. Consequently, discrete QUC requires some modification different from its continuum counterpart. For example, when the potential in \eqref{eq:stationary-Schrodinger} vanishes, Buhovsky et al.~\cite{BLMS} showed that nontrivial harmonic functions on $\Z^2$ have a certain polynomial structure and cannot remain bounded on a sufficiently large full-dimensional portion of $\Z^2$. The techniques of \cite{BLMS} were later employed by Ding and Smart~\cite{DS20} to establish a probabilistic QUC estimate for \eqref{eq:stationary-Schrodinger} on $\Z^2$. The result of \cite{DS20} was subsequently strengthened by Li~\cite{Li22} using the Bourgain--Tzafriri restricted invertibility theory. For more general graphs, Bou-Rabee, Cooperman, and Ganguly~\cite{BAW25} established unique continuation on planar graphs.

When $n=3$, Li and Zhang~\cite{LZ22} proved that solutions of \eqref{eq:stationary-Schrodinger} on $\Z^3$ satisfy a QUC estimate of the form
\begin{equation*}
 \#\left\{x\in Q_L:
 |u(x)|\ge
 \exp\!\left(-C L^{3}\right)|u(0)|
 \right\}
 \ge c\frac{L^{2}}{\log(2+L)}.
\end{equation*}
Their proof again builds on techniques from \cite{BLMS} and makes essential use of the special geometric structure of $\Z^3$.

By contrast, comparatively few results are available for $\Z^n$ when $n\geq 4$. Krymskii~\cite{Kry24} proved that every nonzero solution of a discrete stationary Schr\"odinger equation on $\Z^n$ has support of discrete dimension at least $\log_2 n-7$. Another result, due to Liu, Shi, and Zhang~\cite[Theorem 2.6]{LSZ26}, shows that if the potential contains Bernoulli random variables, then, outside an exceptional event whose probability depends on the initial data, a solution of the Cauchy problem associated with \eqref{eq:stationary-Schrodinger} satisfies
\begin{equation*}
 \#\left\{x\in Q_L:
 |u(x)|\ge
 \exp\!\left(-C L\right)|u(0)|
 \right\}
 \ge cL^{n}.
\end{equation*}
Very recently, Li~\cite{Li26support} established a dimension-reduction principle for support-type estimates in this problem and constructed, in arbitrary dimension $n$, sparse solutions supported on sets of dimension $\lceil n/2\rceil$.

In the line of research described above, the simplex lattice, which can be obtained by intersecting an integer lattice with a hyperplane, plays a central role. It is therefore also important to study unique continuation directly on the simplex lattice. Very recently, Li~\cite{Li26} established support-type, rather than quantitative, unique continuation on the simplex lattice using the Pascal uncertainty principle. We emphasize that the method developed in that work provides important inspiration for the present paper.

Finally, QUC is closely connected with the localization problem for the Anderson--Bernoulli model (ABM) in spectral theory, a connection first pointed out in \cite{BK05}. Several works, including \cite{Jit07,DS20,LZ22,Li22,LSZ26,LSZ26b}, were motivated by this original localization problem. To prove localization for the ABM on $\Z^n$ by multiscale analysis, one needs a QUC estimate of the form
\begin{equation*}
 \#\left\{x\in Q_L:
 |u(x)|\ge
 \exp\!\left(-C L^{a}\right)|u(0)|
 \right\}
 \ge cL^{b},
\end{equation*}
with exponents satisfying
\[
a(a-1)<b/d-1/2.
\]
In the present paper, however, we prove only a QUC estimate with
\[
a=2,
\qquad
b=\lceil n/2\rceil.
\]
Here $\lceil n/2\rceil$ is the smallest integer larger than $n/2$. These exponents lie outside the parameter range required for the localization argument. In this sense, our result remains far from what would be needed for a complete proof of Anderson--Bernoulli localization in high dimensions. Nevertheless, we hope that the techniques developed here will help pave the way toward resolving this localization problem.

\section{Main results}\label{sec:main results}

\subsection{QUC on the simplex lattice}
For $N\in\mathbb Z_{\geq0}$ and dimension $n\geq 3$, we consider the simplex lattice
\begin{equation}\label{eq:simplex-definition}
 \Delta_N^{(n)}
 :=\left\{\alpha=(\alpha_1,\ldots,\alpha_n)\in\mathbb Z_{\geq0}^n:
 |\alpha|:=\sum_{i=1}^n\alpha_i=N\right\}.
\end{equation}
Let $e_i$ be the $i$th coordinate vector and let
$\one_k=(1,\ldots,1)\in\mathbb Z^k$.  Our main result is the following generalization of 
\cite[Theorem 1.9]{LZ22} in high dimensions, which is also a quantitative version of \cite[Theorem 2.1]{Li26}.

\begin{thm}\label{thm:robust-main}
For every fixed $n\geq3$ there exist constants $c_n,C_n>0$ such that
the following holds for every $R\geq1$.  Let
$g:\Delta_{nR}^{(n)}\to\mathbb R$ satisfy $g(R\one_n)\neq0$ and
\begin{equation}\label{eq:robust-main-clean-assumption}
 \left|\sum_{i=1}^ng(\beta+e_i)\right|
 \leq \econst^{-C_nR}|g(R\one_n)|, \ \forall \beta\in\Delta_{nR-1}^{(n)}.
\end{equation}
Then
\begin{equation}\label{eq:robust-main-clean-conclusion}
 \#\left\{\alpha\in\Delta_{nR}^{(n)}:
 |g(\alpha)|\geq \econst^{-C_nR}|g(R\one_n)|\right\}
 \geq c_nR^{\lceil n/2\rceil}.
\end{equation}
\end{thm}

Theorem \ref{thm:robust-main} is indeed a special case of the following unique continuation result, in which we do not assume the large value taken in the center:
\begin{thm}
\label{thm:robust-anisotropic}
For every fixed $n\geq 3$ there exist constants $c_n,C_n>0$ such that
the following holds for every $N\geq 1$ and $a\in \Delta^{(n)}_{N}$. Let $g:\Delta_N^{(n)}\to\R$ satisfying
$g(a)\ne 0$ and
\begin{equation}\label{eq:robust-simplex-residual}
 \left|\sum_{i=1}^ng(\beta+e_i)\right|
 \leq \econst^{-C_n N}|g(a)|, \ \forall \beta\in\Delta_{N-1}^{(n)}.
\end{equation}
Here, without loss of generality, up to a coordinate permutation we can assume that 
\begin{equation}\label{eq:ordered-a}
 a_1\ge a_2\ge\cdots\ge a_n.
\end{equation}
 Then
\begin{equation}\label{eq:robust-anisotropic-large}
 \#\left\{\alpha\in\Delta_{N}^{(n)}:
 |g(\alpha)|\geq \econst^{-C_n N}|g(a)|\right\}
 \geq c_n\Phi_n(a),
\end{equation}
with 
\begin{equation}\label{eq:anisotropic-weight}
 \qquad
 \Phi_n(a):=
 \frac{(a_n+1)\prod_{i=1}^{n-2}(a_i+1)}{(N+1)^{h_0}},
 \qquad
 h_0=\left\lfloor\frac{n-2}{2}\right\rfloor.
\end{equation}
\end{thm}

\begin{proof}[Proof of Theorem 2.1]
	Take $a=R\one_n$ and $N=nR$ in Theorem \ref{thm:robust-anisotropic}, \eqref{eq:robust-main-clean-conclusion} will be implied by 
	\[  \Phi_n(R\one_1)= \frac{(R+1)^{n-1}}{(nR+1)^{h_0}}\sim R^{n-1-h_0}=R^{\lceil n/2 \rceil}. \]
\end{proof}

Actually, we will show that Theorem \ref{thm:robust-anisotropic} can be deduced from the following weaker form by replacing \eqref{eq:robust-simplex-residual} by an exactly reccurence relationship:

\begin{thm}[Exact-recurrence QUC]
\label{thm:main-anisotropic-exact}
For every fixed $n\geq 3$ there exist constants $c_n,C_n>0$ such that
the following holds for every $N\geq 1$ and $a\in \Delta^{(n)}_{N}$. Let $g:\Delta_N^{(n)}\to\R$ satisfying
$g(a)\ne 0$ and
\begin{equation}\label{eq:anisotropic-exact}
  \sum_{i=1}^ng(\beta+e_i)=0 , \ \forall \beta\in\Delta_{N-1}^{(n)}.
\end{equation}
Assuming also \eqref{eq:ordered-a}, then \eqref{eq:robust-anisotropic-large} holds.

\end{thm}

\begin{rmk}
We do not consider the two dimensional case, i.e. $n=2$, since it is very simple and not as complicated as the $n\geq 3$ case. Indeed, if $g:\Delta_{N}^{(2)}\to\mathbb R$ satisfies
\eqref{eq:anisotropic-exact}, then
\begin{equation}\label{eq:n2-exact}
 g(j,N-j)=(-1)^{j-a_1}g(a),
 \qquad 0\leq j\leq2R.
\end{equation}
Therefore $|g(\alpha)|=|g(a)|$ for all $\alpha \in \Delta_{N}^{(2)}$ and \eqref{eq:robust-anisotropic-large} holds automatically.
\end{rmk}

\subsection{QUC on $\Z^n$}
For $x=(x_1,\ldots,x_n)\in\Z^n$, write
\[
 |x|_1=\sum_{i=1}^n|x_i|,
 \qquad
 |x|_\infty=\max_{1\le i\le n}|x_i|,
 \qquad
 Q_L=\{x\in\Z^n:|x|_\infty\le L\}.
\]
Especially, we denote $Q_L=Q_L(0)$. We use the Laplacian convention
\begin{equation}\label{eq:laplacian-convention}
 (\Delta_{\Z^n} u)(x)=\sum_{|y-x|_1=1}u(y)-2n \,u(x).
\end{equation}

We also have the following quantitative unique continuation result on $\Z^n (n\geq 3)$, which generalizes \cite[Theorem 1.5]{LZ22} to arbitrary dimension case and gives affirmative answer to the \cite[Conjecture 1.2]{Li26support} up to a $(\log L)^{-1}$ factor.

\begin{thm}[QUC in a lattice cube]
\label{thm:main}
Let $n \ge3$ and $K\ge0$.  There exist $c_n, C_{n,K}, L_{n,K}>0$ with the following property.  Suppose that $L\ge L_{n,K}$,
$u:Q_L\to \R$, $V:Q_L\to \R$, $\|V\|_{\ell^\infty(Q_L)}\le K$, and
\begin{equation}\label{eq:main-equation}
 \Delta_{\Z^n} u=Vu
 \quad\text{at every }x\in Q_L\text{ whose nearest neighbours lie in }Q_L.
\end{equation}
If $u(0)\ne0$, then
\begin{equation}\label{eq:main-conclusion}
 \#\left\{x\in Q_L:
 |u(x)|\ge
 \exp\!\left(-C_{n,K}L^2 \right)|u(0)|
 \right\}
 \ge c_n\frac{L^{\lceil n/2\rceil}}{\log(2+L)}.
\end{equation}
\end{thm}
We remark that, by \cite[Proposition 2.3]{Li26support}, the cardinality lower bound of $L^{\lceil n/2 \rceil}$ is sharp up to the $(\log L)^{-1}$ factor.

\begin{rmk}
	All the quantitative unique-continuation results above (for both simplex lattice and $\Z^n$)
remain valid, with unchanged feasible constants, for complex-valued $u$ and $V$.
 The complex value case is just of the total same proof, which one can keep in mind and check line by line in this paper.
\end{rmk}

\subsection{Organization of the paper}
The paper is organized as follows. In section \ref{sec:preliminaries}, we introduce some basic knowlegdes. In section \ref{sec:pascal}, we establish several quantitative versions of Pascal uncertainty principle (UP). The Pascal UP will be used in section \ref{sec:prove-QUC-on-simplex} to establish the QUC on simplex lattice. In section \ref{sec:stability}, we deduce Theorem \ref{thm:robust-anisotropic} by studying the stability of Theorem \ref{thm:main-anisotropic-exact}. Finally, in section \ref{sec:prove-QUC-on-lattice}, we use Walsh transform to connect fucntions on simplex with functions on lattice, and prove Theorem \ref{thm:main} by Theorem \ref{thm:robust-anisotropic}.

\subsection{Notations}
The  notations used in this paper can be collected as follows. 
\begin{itemize}
\item For multiindices, we denote
\begin{equation}\label{eq:multiindex-notation}
 \alpha!:=\prod_{i=1}^n\alpha_i!,
 \qquad z^\alpha:=\prod_{i=1}^nz_i^{\alpha_i},
 \qquad \partial^\alpha:=\prod_{i=1}^n\partial_{z_i}^{\alpha_i}.
\end{equation}
\item For $z=(z_1,z_2,\cdots,z_n)$, we use $\R[z]$ to denote set containing all n-variable polynomials with real coefficients, which is a linear real vector space. Moreover, $\R_k[z]$ (resp. $\R_{\leq k}[z]$) stands for the set containing all polynomials with degree $k$ (resp. $\leq k$).  We also denote $\R^{\rm homo}_k[z]$ the set containing all homogeneous polynomials with degree $k$.  For a polynomial $P(z)\in \R[z]$, we use $[z^\alpha]P$ to represent the coefficient of $z^\alpha$ in $P$.
\item $C=C_{a,b,\cdots}$ means the constant $C$ depends on parameters $a,b,\cdots$.
\item We adopt the Vinogradov symbol $f\lesssim g$ for two nonnegative quantities $f$ and $g$, if there is an absolute constant
$C > 0$ such that $f\leq Cg$. If we want to emphasize that $C$ depends on some parameters
$a,b,\cdots$ independent of $f,g$, then we write $f\lesssim_{a,b,\cdots}g$. We denote $f\sim g$ if  both $f\lesssim g$ and $g\lesssim f$. In some cases, we denote  $f\ll g$  if there is some small  enough $c>0$ independent of $f,g$ so that $f\leq c g.$ 
\item We adapt the Landau symbol $f=O(g)$ to mean that $|f|\lesssim g$ and $f=O_{a,b,\cdots}(g)$ to mean that $|f|\lesssim_{a,b,\cdots} g$.  We also write $f=o(g)$ to mean $|f|\ll g$.
\item For a set $S$, we denote by $\# S$ or $|S|$ the cardinality of it. 
\item For any $a\in\R$, we denote $\lfloor a \rfloor$ the integer part of $a$, and $\lceil a \rceil$ the smallest integer that larger than $a$.
\item We use $|\cdot|_{\infty}$ to denote the $\ell^{\infty}$-norm and $|\cdot|_1$ the $\ell^1$-norm for vectors in $\Z^n$ or $\R^n$. Moreover, we write $Q_L(x)=\{y\in \Z^n : |y-x|_{\infty} \le L\}$ and set $Q_L = Q_L(0)$. 

\end{itemize}

\section{Preliminaries}\label{sec:preliminaries}

\subsection{Factorial normalization and the polynomial dictionary}\label{subsection:factorial-normalization}
In this part, we connect the values of $g$ with coefficients of some polynomials. For every $g:\Delta_{nR}^{(n)}\to\mathbb R$, we define the homogeneous polynomial of degree $nR$
\begin{equation}\label{eq:F-definition-detailed}
 F(z_1,\ldots,z_n)
 :=\sum_{\alpha\in\Delta_{nR}^{(n)}}
 g(\alpha)\frac{z^\alpha}{\alpha!}
\end{equation}
and the directional derivative
\begin{equation}\label{eq:D-definition}
 D:=\partial_{z_1}+\cdots+\partial_{z_n}.
\end{equation}

\begin{lem}[Recurrence and differential-equation equivalence]
\label{lem:recurrence-D}
The relations \eqref{eq:anisotropic-exact} are equivalent to
\begin{equation}\label{eq:DF-zero-detailed}
 DF=0.
\end{equation}
\end{lem}

\begin{proof}
Simple computation will show that 
\begin{equation}\label{eq:D-coefficient}
 [z^\beta]DF
 =\frac1{\beta!}\sum_{i=1}^ng(\beta+e_i),\ \forall \beta\in \Delta_{nR-1}^{(n)}.
\end{equation}
Since every coefficient of $DF$ is indexed by a $\beta\in  \Delta_{nR-1}^{(n)}$, 
\eqref{eq:D-coefficient} proves the equivalence.
\end{proof}

To study the fucntion $g$, it's equavalent to study the coefficient of polynomials satisfying \eqref{eq:DF-zero-detailed}, i.e. the polynomials lie in the kernel of $D$. For those polynomials, one may easily check that we have the following properties: 
\begin{prop}[Kernel of $D$]
\label{lem:diagonal-invariance}
For a polynomial $H\in\mathbb R[z_1,\ldots,z_n]$, the following statements are equivalent:
\begin{itemize}
	\item[(1)] $H\in \ker(D)$;
	\item[(2)] $H(z+s\one_n)=H(z), \ \forall z\in\mathbb R^n,\ \forall s\in\mathbb R$;
	\item[(3)] $H(z)=B(z_1-z_n,\ldots,z_{n-1}-z_n)$ for a unique polynomial $B$. Moreover, if $H$ is homogeneous of degree $L$, then $B$ is also homogeneous of degree $L$.
\end{itemize}
\end{prop}

\subsection{Newton bases and Pascal matrices}
For $k\geq0$, define the following $k$-degree polynomials about variable $x\in \R$:
\begin{equation}\label{eq:falling-factorial}
 \begin{aligned}
 (x)_k&:=x(x-1)\cdots(x-k+1),&
 (x)_0&:=1.
\end{aligned}
\end{equation}
Also, imitating the definition of combination number, we define 
\begin{equation}
	 \binom{x}{k}:=\frac{(x)_k}{k!}, \ k\geq0; \quad \binom{x}{k}:=0,\quad k<0.
\end{equation}
We call $\left\{\binom{x}{k} \right\}_{k\geq 0}$ the \textbf{Newton basis} of the one-variable polynomial space $\R[x]$.

Let
\begin{equation}\label{eq:finite-difference}
 \Delta:\R_k[x]\rightarrow \R_{k-1}[x]     ;\quad (\Delta p)(x):=p(x+1)-p(x)
\end{equation}
be the first-order difference operator. Elementary computation shows that 
\begin{equation}\label{eq:delta-binomial}
 \Delta\binom{x}{k}=\binom{x}{k-1}.
\end{equation}

\begin{lem}[Newton expansion]\label{lem:newton-expansion}
Every polynomial $P\in \R[x]$ of degree at most $L$ has the unique expansion
\begin{equation}\label{eq:newton-expansion}
 P(x)=\sum_{k=0}^L d_k\binom{x}{k},
\end{equation}
with the coefficient
\begin{equation}\label{eq:finite-difference-inversion}
 d_k=\Delta^kP(0)=\sum_{u=0}^k(-1)^{k-u}\binom kuP(u).
\end{equation}
\end{lem}

\begin{proof}
Since $\binom{x}{k}$ has degree $k$ and leading coefficient $1/k!$,
the polynomials $\binom{x}{0},\ldots,\binom{x}{L}$ form a basis of $\R_{\leq L}[x]$.
Iterating \eqref{eq:delta-binomial} gives
\begin{equation}\label{eq:iterated-delta-binomial}
 \Delta^j\binom{x}{k}=\binom{x}{k-j}.
\end{equation}
At $x=0$, the right side is $1$ if $k=j$ and $0$ if $k>j$.
Therefore, taking $x=0$ in 
\[\Delta^j P(x)=\sum_{k=j}^L d_k \binom{x}{k-j} \]
yields the first equivalence of \eqref{eq:finite-difference-inversion}. The second equivalence follows by expanding
$\Delta^k=(E-I)^k$, where $(Ep)(x):=p(x+1)$.
\end{proof}

Since we want to study $\Delta$, we need to investigate the shift of polynomials by one. Let
\begin{equation}\label{eq:q-and-shift}
 q(y)=\sum_{j=0}^La_jy^j,
 \qquad
 q(y+1)=\sum_{i=0}^Lb_iy^i.
\end{equation}
The binomial theorem gives the relationship
\begin{equation}\label{eq:ordinary-Pascal-transform}
 b_i=\sum_{j=i}^L\binom ji a_j.
\end{equation}
We denote the \textbf{Pascal upper-diagonal matrix} 
\begin{equation}\label{eq:Pascal-matrix}
	\mcP_L=\left(  \binom{j}{i} \right)_{1\leq i,j\leq L}.
\end{equation}
Then \eqref{eq:ordinary-Pascal-transform} becomes $b=\mcP_L a$. Moreover, under the Newton expansion \eqref{eq:newton-expansion}, we have 
\begin{equation}\label{eq:newton-Pascal-evaluation}
 P(j)=\sum_{k=0}^j\binom jk d_k=(\mcP_L^{\top} d)_j.
\end{equation}

Given strictly increasing integer sequences
\begin{equation}\label{eq:pascal-minor-sequences}
 0\leq r_1<\cdots<r_q,
 \qquad 0\leq k_1<\cdots<k_q,
\end{equation}
the corresponding minor of Pascal matrix $[\binom{r_i}{k_j}]_{i,j=1}^q$ is called \textbf{proper} if
\begin{equation}\label{eq:proper-minor-definition}
 k_i\leq r_i,\quad \forall \ 1\leq i\leq q.
\end{equation}


\subsection{The Lindstr\"om--Gessel--Viennot determinant Theorem}

We also need the following Theorem to compute the determinant of edge-weight transport matrix on finite directed acyclic graphs, which originates from  
\cite{Lin73,GV85}.

\begin{thm}[LGV]\label{thm:LGV}
Let $\mathcal G$ be a finite directed acyclic graph with edge weights in a
commutative ring.  For startings point $S_1,\ldots,S_q$ and terminations
$T_1,\ldots,T_q$, let $M_{ij}$ be the sum of the weights of all directed
paths from $S_i$ to $T_j$.  Then
\begin{equation}\label{eq:LGV-formula}
 \det M=\sum_{\sigma\in\mathfrak S_q}\operatorname{sgn}(\sigma)
 \sum_{\substack{(P_1,\ldots,P_q)\text{ vertex-disjoint}\\
                  P_i:S_i\to T_{\sigma(i)}}}
 \prod_{i=1}^q\operatorname{wt}(P_i).
\end{equation}
Here $\mathfrak S_q$ is the $q$-order permutation group and $\operatorname{wt}(P_i)$ stands for the weight of the path $P_i$.
\end{thm}

\begin{proof}[Proof of Theorem \ref{thm:LGV}](Sketch)
Expand the determinant and then every matrix entry:
\begin{equation}\label{eq:LGV-expansion}
 \det M
 =\sum_{\sigma\in\mathfrak S_q}\operatorname{sgn}(\sigma)
 \sum_{P_i:S_i\to T_{\sigma(i)}}\prod_i\operatorname{wt}(P_i).
\end{equation}
For an intersecting path family, choose the first intersection according
to a fixed topological ordering of the vertices, and among the paths
through it choose the lexicographically first pair $P_i,P_j$.  Exchange
their tails after that vertex.  The product of edge weights is unchanged,
whereas the terminal permutation is composed with the transposition
$(i\ j)$, so its sign changes and  cause a cancel in \eqref{eq:LGV-expansion}. The
remaining terms give \eqref{eq:LGV-formula}.
\end{proof}


\subsection{The Young diagram and the Young tableau}
We introduce the concept of partition of non-negative integers, the Young diagram and the Young tableau.

\begin{Def}
A \textbf{partition} of length at most $q$ is a sequence
\begin{equation}\label{eq:partition-definition}
 \lambda=(\lambda_1,\ldots,\lambda_q) \in \Z_{\geq 0}^q,
 \qquad
 \lambda_1\geq\cdots\geq\lambda_q\geq0.
\end{equation}
Its \textbf{Young diagram} is
\begin{equation}\label{eq:young-diagram}
 [\lambda]:=\{(i,j):1\leq i\leq q,\ 1\leq j\leq\lambda_i\}.
\end{equation}
We rule that row
indices increase from top to bottom, and column indices increase from
left to right. 
A \textbf{semistandard Young tableau} of $\lambda$ is a map
$U:[\lambda]\to[q]$ satisfying
\[
 U(i,j)\leq U(i,j+1),
 \qquad U(i,j)<U(i+1,j)
\]
whenever the displayed boxes belong to $[\lambda]$.  A \textbf{reverse
semistandard tableau} $T:[\lambda]\to[q]$ satisfies instead
\[
 T(i,j)\geq T(i,j+1),
 \qquad T(i,j)>T(i+1,j).
\]
We denote corresponding sets of semistandard Young tableau and reverse
semistandard tableau by
$\operatorname{SSYT}_q(\lambda)$ and
$\operatorname{RSSYT}_q(\lambda)$. 
\end{Def}
One can see Figure \ref{fig:young-diagram-tableau} to have a better understanding of the above definition. Moreover, for partitions $\lambda,\mu$, write
\begin{equation}\label{eq:partition-inclusion}
 \lambda\subseteq\mu
 \ \Longleftrightarrow\
 \lambda_i\leq\mu_i,\quad \forall \ 1\leq i\leq q.
\end{equation}

\begin{figure}[htbp]
  \centering
  \begin{tikzpicture}[
      x=0.8cm,
      y=0.8cm,
      line cap=round,
      line join=round,
      every node/.style={font=\small}
    ]

    \begin{scope}[shift={(0,0)}]

      \draw[->,thick] (-0.35,0.35) -- (1.15,0.35)
        node[right] {$j$};
      \draw[->,thick] (-0.35,0.35) -- (-0.35,-1.15)
        node[below] {$i$};

      \foreach \j in {0,1,2,3}
        \draw (\j,0) rectangle ++(1,-1);

      \foreach \j in {0,1,2}
        \draw (\j,-1) rectangle ++(1,-1);

      \foreach \j in {0,1,2}
        \draw (\j,-2) rectangle ++(1,-1);

      \draw (0,-3) rectangle ++(1,-1);

      \node[font=\normalsize] at (2,-4.55)
        {Young diagram: $\lambda=(4,3,3,1)$};
    \end{scope}

    \begin{scope}[shift={(7,0)}]

      \draw[->,thick] (-0.35,0.35) -- (1.15,0.35)
        node[right] {$j$};
      \draw[->,thick] (-0.35,0.35) -- (-0.35,-1.15)
        node[below] {$i$};

      \foreach \j/\entry in {0/1,1/1,2/5,3/9} {
        \draw (\j,0) rectangle ++(1,-1);
        \node at (\j+0.5,-0.5) {\entry};
      }

      \foreach \j/\entry in {0/3,1/3,2/7} {
        \draw (\j,-1) rectangle ++(1,-1);
        \node at (\j+0.5,-1.5) {\entry};
      }

      \foreach \j/\entry in {0/6,1/8,2/10} {
        \draw (\j,-2) rectangle ++(1,-1);
        \node at (\j+0.5,-2.5) {\entry};
      }

      \draw (0,-3) rectangle ++(1,-1);
      \node at (0.5,-3.5) {11};

      \node[font=\normalsize] at (2,-4.55)
        {semistandard Young tableau};
    \end{scope}

  \end{tikzpicture}
  \caption{The corresponding Young diagram and a semistandard Young tableau of partition $\lambda=(4,3,3,1)$.}
  \label{fig:young-diagram-tableau}
\end{figure}
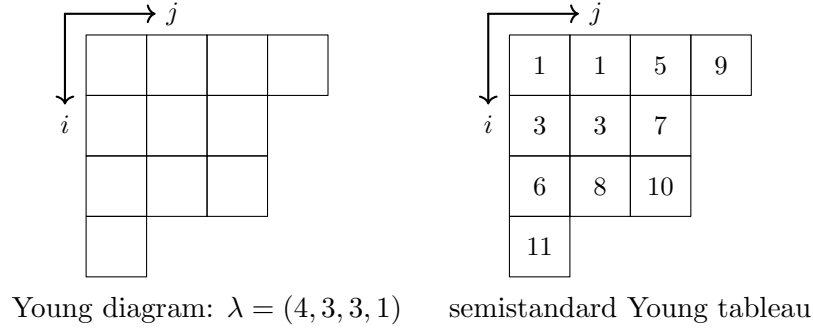

\section{Quantitative Pascal uncertainty principle}\label{sec:pascal}
In this section, we establish a quantitative version of Pascal uncertainty principle about coefficients of polynomials. The qualitative Pascal uncertainty principle is a key ingredient in \cite{Li26}.

\subsection{Shifted Schur polynomials and several polynomial interpolation results}
We first introduce some combinatorial results which will be useful in our proof.

Let $\lambda$ be a partition of length at most $q$. Define the \textbf{Shifted Schur polynomial} $s^*_{\lambda}(x_1,\ldots,x_q)$ as 
 \begin{equation}\label{eq:shifted-Schur-polynomial}
	 s^*_{\lambda}(x_1,\ldots,x_q)
 :=\frac{\det[(x_i+q-i)_{\lambda_j+q-j}]_{1\leq i,j\leq q}}
         {\det[(x_i+q-i)_{q-j}]_{1\leq i,j\leq q}},
 \qquad
 (u)_d:=u(u-1)\cdots(u-d+1).
 \end{equation}

\begin{lem}[Monotonicity of shifted Schur polynomial.]
\label{lem:shifted-schur-monotone}
Let $q\geq1$, and let $\lambda=(\lambda_1,\ldots,\lambda_q)$ be a
partition.  If $\mu,\nu$ are partitions of length at most $q$ and $\mu\subseteq \nu$, then
\[
 0\leq s^*_{\lambda}(\mu)\leq s^*_{\lambda}(\nu).
\]
\end{lem}
\begin{proof}[Proof of Lemma \ref{lem:shifted-schur-monotone}]
	We first prove that 
\begin{equation}\label{eq:reverse-tableau}
 s^*_{\lambda}(x_1,\ldots,x_q)
 =\sum_{T\in\operatorname{RSSYT}_q(\lambda)}
   \prod_{(i,j)\in[\lambda]}
   \bigl(x_{T(i,j)}-(j-i)\bigr),
\end{equation}
which indeed comes from \cite[Theorem 11.1]{OO97}. \eqref{eq:reverse-tableau} will be deduced from the following Factorial bialternant--tableau identity.
\begin{lem}[Factorial bialternant--tableau identity]
\label{lem:factorial-tableau}
Let $a=(a_1,a_2,\ldots)$ be arbitrary real sequence and $y\in \R$. We set
\[
 (y\mid a)^0:=1,
 \qquad
 (y\mid a)^d:=\prod_{h=1}^d(y-a_h)\quad(d\geq1).
\]
For every partition $\lambda$ of length at most $q$,
\begin{equation}\label{eq:factorial-tableau}
 \frac{\det[(y_i\mid a)^{\lambda_j+q-j}]_{1\leq i,j\leq q}}
      {\det[(y_i\mid a)^{q-j}]_{1\leq i,j\leq q}}
 =\sum_{U\in\operatorname{SSYT}_q(\lambda)}
   \prod_{(i,j)\in[\lambda]}
   \bigl(y_{U(i,j)}-a_{U(i,j)+j-i}\bigr).
\end{equation}
\end{lem}
For the completeness of the paper, we delay the proof of Lemma \ref{lem:factorial-tableau} in Appendix $A$.\\

Now We derive \eqref{eq:reverse-tableau} from
Lemma~\ref{lem:factorial-tableau}.  In \eqref{eq:factorial-tableau} set
\[
 a_h:=h-1,
 \qquad
 y_u:=x_{q-u+1}+u-1.
\]
Then $(y_u\mid a)^d=(x_{q-u+1}+u-1)_d$.  Reversing the rows (i.e., take $q-u+1=i$) in both numerator and denominater of
 \[\frac{\det[(x_{q-u+1}+u-1 )_{\lambda_j+q-j}]_{1\leq u,j\leq q}}
      {\det[(x_{q-u+1}+u-1)_{q-j}]_{1\leq u,j\leq q}},\]
it exactly equals
\[
 \frac{\det[(x_i+q-i)_{\lambda_j+q-j}]_{i,j=1}^q}
      {\det[(x_i+q-i)_{q-j}]_{i,j=1}^q}
 =s^*_{\lambda}(x_1,\ldots,x_q).
\]
Therefore, 
\begin{equation}\label{eq:shifted-schur-LGV}
	s^*_{\lambda}(x_1,\ldots,x_q)=\sum_{U\in\operatorname{SSYT}_q(\lambda)}
   \prod_{(i,j)\in[\lambda]}
   \bigl(y_{U(i,j)}-a_{U(i,j)+j-i}\bigr).
\end{equation}

For $U\in\operatorname{SSYT}_q(\lambda)$ define
$T(i,j):=q+1-U(i,j)$.It is easy to check that such transform from $U$ to $T$ a bijection form $\operatorname{SSYT}_q(\lambda)$ onto
$\operatorname{RSSYT}_q(\lambda)$.  Moreover,
\[
 \begin{aligned}
 y_{U(i,j)}-a_{U(i,j)+j-i}
 &=x_{q+1-U(i,j)}+U(i,j)-1
   -(U(i,j)+j-i-1)\\
 &=x_{T(i,j)}-(j-i).
 \end{aligned}
\]
Substituting this into \eqref{eq:shifted-schur-LGV} proves \eqref{eq:reverse-tableau}.

It remains to prove monotonicity.  Fix
$T\in\operatorname{RSSYT}_q(\lambda)$ and put
\[
 W_T(x):=\prod_{(i,j)\in[\lambda]}
       \bigl(x_{T(i,j)}-(j-i)\bigr).
\]
We show first that $W_T(\mu)$ is either zero or strictly positive.
Suppose it is nonzero.  Along the first row set
$b_j:=\mu_{T(1,j)}$.  Since $T(1,j+1)\leq T(1,j)$ and $\mu$ is
nonincreasing, $b_{j+1}\geq b_j$.  The first factor is nonzero, so
$b_1\geq1$.  If $b_j\geq j$, then $b_{j+1}\geq j$; nonvanishing of the
factor at $(1,j+1)$ says $b_{j+1}\neq j$, hence
$b_{j+1}\geq j+1$.  Induction gives
\[
 \mu_{T(1,j)}-(j-1)\geq1
\]
for every box in the first row.  If $(i,j)$ lies below $(i-1,j)$, strict
decrease down columns gives $T(i,j)<T(i-1,j)$ and therefore
\[
 \mu_{T(i,j)}-(j-i)
 \geq \mu_{T(i-1,j)}-(j-(i-1))+1.
\]
Induction down the column shows that every factor of $W_T(\mu)$ is
positive.  Thus $W_T(\mu)\geq0$.

If $W_T(\mu)>0$, replacing $\mu$ by the coordinatewise larger partition
$\nu$ can only increase every positive factor, so
$W_T(\nu)\geq W_T(\mu)$.  If $W_T(\mu)=0$, the same argument applied to
$\nu$ gives $W_T(\nu)\geq0=W_T(\mu)$.  Summing over $T$ in
\eqref{eq:reverse-tableau} proves the lemma.

\end{proof}

Using the monotonicity of the shifted Schur polynomial, we can prove the following polynomial interpolation theorem:
\begin{lem}
\label{lem:pascal-cardinal}
Let
\begin{equation}\label{eq:proper-Pascal-index-interpolation}
	 0\leq r_1<\cdots<r_q\leq L,
 \qquad 0\leq k_1<\cdots<k_q\leq L,
 \qquad k_i\leq r_i.
\end{equation}

Fix $m\in\{r_1,\ldots,r_q\}$.  There is a unique polynomial
$F\in\operatorname{span}\{\binom{x}{k_i}:1\leq i\leq q\}\subset \R_{\leq L}[x]$ such that
$F(m)=1$ and $F(r)=0$ for all other $r\in\{r_1,\ldots,r_q\}$. Moreover, for every
integer $t$ with $0\leq t\leq m$, we have
\begin{equation}\label{eq:cardinal-product-bound}
 |F(t)|\leq
 \prod_{r\in\{r_1,\ldots,r_q\}\setminus\{m\}}
       \frac{|t-r|}{|m-r|}.
\end{equation}
\end{lem}

\begin{proof}
By \eqref{eq:proper-Pascal-index-interpolation}, the Pascal minor $P=(\binom{r_i}{k_j})_{i,j=1}^q$ of $\mcP_L$ is proper. It was shown in the proof of \cite[Lemma 4.1]{Li26} that the determinant of proper Pascal minor is strictly positive (which also need to leverage Theorem \ref{thm:LGV}). Therefore $\det P\geq 1$.

Write $\mathbf r=(r_1,\ldots,r_q)$ and define partitions
\[
 \lambda_i=k_{q-i+1}-(q-i),\qquad
 \mu_i=r_{q-i+1}-(q-i)\qquad(1\leq i\leq q).
\]
Obviously, 
$\lambda\subseteq\mu$ because $k_i\leq r_i$. Recalling notations in \eqref{eq:shifted-Schur-polynomial}, let $J_q$ be the
$q\times q$ reversal matrix (i.e. $(J_q)_{ij}=\delta_{q+1,i+j}$ for $1\leq i,j\leq 1$) and put
\[
 A:=\bigl[(r_i)_{k_j}\bigr]_{i,j=1}^q,
 \qquad
 B:=\bigl[(r_i)_{j-1}\bigr]_{i,j=1}^q,
 \qquad
 V(\mathbf r):=\prod_{1\leq i<j\leq q}(r_j-r_i).
\]
Because
\[
 \mu_i+q-i=r_{q-i+1},
 \qquad
 \lambda_j+q-j=k_{q-j+1},
\]
the numerator and denominator alternants in the definition of
$s^*_{\lambda}(\mu)$ are $J_qAJ_q$ and $J_qBJ_q$, respectively.  The
square reversal determinants $(\det J_q)^2$ cancel. Therefore, 
\[ s^*_{\lambda}(\mu)= \frac{\det A}{\det B}.\]
It's not hard to observe that up to some column operations, $\det(B)$ is exactly the Vandermonde determinant of variables $\mathbf r=(r_1,\ldots,r_q)$, i.e. $\det B=V(\mathbf r)$. Moreover, since 
\[
 \binom{r_i}{k_j}
 =(r_i)_{k_j}  / k_j !,
\]
we have 
\begin{equation*}
 \det\!\left[\binom{r_i}{k_j}\right]_{i,j=1}^q
 =\frac{\det A}
        {\prod_{j=1}^q k_j!}.
\end{equation*}
Consequently,
\begin{equation}\label{eq:pascal-shifted-schur}
 \det P= \det\!\left[\binom{r_i}{k_j}\right]_{i,j=1}^q
 =\frac{\prod_{1\leq i<j\leq q}(r_j-r_i)}
        {\prod_{j=1}^q k_j!}\,s^*_{\lambda}(\mu).
\end{equation}
In particular, $\det P\geq 1$ and $V(\mathbf r)>0$; consequently
$s^*_{\lambda}(\mu)>0$, so the quotient used below is legitimate.

Now write 
\[F(x)=\sum_{1\leq j\leq q} s_j \binom{x}{k_j}.\]
Denote $\mathbf s=(s_1,\cdots ,s_q)$, the condition $F(m)=1$ and $F(r_i)=0,r_i\neq m$ means that $P\mathbf s=\delta_m$. Since $\det P\geq 1$, such $\mathbf s$ has a unique solution and therefore the feasible polynomial $F(x)$ exists and is unique. By Cramer's rule, $s_j$ is the quotient of the determinant obtained
by replacing the $j$-column with $\delta_m$ in $P$, and the
original determinant $\det P$. The numerator of $s_j$ indeed is the $(m,j)$-algebraic-cofactor of $\det P$, so the Laplace theorem of determinants ensures that
\begin{equation}
	F(t)=\frac{\det P^{(t)}}{\det P},
\end{equation}
where $P^{(t)}$  is the matrix obtained by replacing the row vector $(\binom{m}{k_j})_{1\leq j \leq q}^{\top}$ with $(\binom{t}{k_j})_{1\leq j \leq q}^{\top}$ in $P$.

If $t=m$ , both hand sides of \eqref{eq:cardinal-product-bound} equals 1. If $1\leq t <m$ and $t\in \{r_1,\cdots,r_q\}\setminus\{m\}$, the numerator is
zero and \eqref{eq:cardinal-product-bound} is immediate.  Otherwise, we replace $m$ by $t<m$ in $\mathbf r$, and rearrange it to obtain an strictly increasing sequence $\mathbf r^{(t)}$.
The associated partition
$\mu^{(t)}$ satisfies $\mu^{(t)}\subseteq \mu$.  If the new Pascal minor is not
proper, its determinant is zero.  Indeed, for $\mathbf r^{(t)}$ contains
$r'_1<\cdots<r'_q$, choose $i$ with $r'_i<k_i$.  If $a\leq i$ and
$b\geq i$, then $r'_a\leq r'_i<k_i\leq k_b$; hence the first $i$ rows
vanish in the last $q-i+1$ columns and are supported on only $i-1$
columns.  They are linearly dependent and therefore $\det P^{(t)}=0$.  If the new minor is proper, then
$\lambda\subseteq\mu^{(t)}$, and Lemma~\ref{lem:shifted-schur-monotone}
and \eqref{eq:pascal-shifted-schur} give
\[
 |F(t)|
 =\left|\frac{V(\mathbf r^{(t)})}{V(\mathbf r)}
   \frac{s^*_{\lambda}(\mu^{(t)})}{s^*_{\lambda}(\mu)}\right|
 \leq\left|\frac{V(\mathbf r^{(t)})}{V(\mathbf r)}\right|
 =\prod_{r\in\{r_1,\ldots,r_q\}\setminus\{m\}}
   \frac{|t-r|}{|m-r|}.
\]
This is \eqref{eq:cardinal-product-bound}.
\end{proof}

With the help of Lemma \ref{lem:pascal-cardinal}, we have the following interpolation theorem to match $F(x)$ with initial data about both values and coefficients under Newton basis of $F(x)$.

\begin{thm}
\label{thm:mixed-decoder}
Let $0\leq m\leq L$, let $J,E\subseteq\{0,1,\ldots,L\}$, assume
$m\in J$, and suppose
\begin{equation}\label{eq:J-E-large}
	        |J|+|E|\leq m+1.
\end{equation}
Then there is a polynomial $F(x)\in R_{\leq m}[x]$  such that
\[
 F(m)=1,\ F(j)=0\quad \forall j\in J\setminus\{m\};
 \quad \Delta^eF(0)=0\quad \forall e\in E.
\]
Moreover, by the Newton expansion, writing
\[
 F(x)=\sum_{i=0}^m d_i\binom{x}{i},
 \qquad d_i=\Delta^iF(0),
\]
one may arrange $F(x)$ to satisfy
\begin{align}
 \max_{0\leq t\leq m}|F(t)|&\leq(2\econst)^L,\label{eq:decoder-inner}\\
 \max_{0\leq t\leq L}|F(t)|&\leq(6\econst)^L,\label{eq:decoder-outer}\\
 \sum_{i=0}^m|d_i|&\leq 2(4\econst)^L.\label{eq:decoder-l1}
\end{align}
\end{thm}

\begin{proof}
For $-1\leq t\leq m$ put
\[
 W(-1):=0,\qquad
 W(t):=t+1-|J\cap[0,t]|-|E\cap[0,t]|.
\]
Choose $c-1\in\{-1,0,\ldots,m-1\}$ at which $W$ attains the minimum of
$W$ on $\{-1,0,\ldots,m\}$; such a choice is possible because
$W(-1)=0$ and by \eqref{eq:J-E-large},
\[
 W(m)=m+1-|J\cap[0,m]|-|E\cap[0,m]|
 \geq |J\cap[m+1,L]|\geq0.
\]
Set
\[
 J':=J\cap[c,L],\qquad
 A:=\{c,c+1,\ldots,m\}\setminus E,\qquad q:=|J'|.
\]
It's easy to check that
\begin{align*}
 |A|-|J'|
 &=W(m)-W(c-1)-|J\cap[m+1,L]|\geq0.
\end{align*}
Moreover, for every $c\leq t\leq m$,
\begin{align}\label{eq:minimum-of-W}
 |A\cap[c,t]|-|J'\cap[c,t]|
 &=W(t)-W(c-1)\geq0.
\end{align}
Let $K=\{k_1<\cdots<k_q\}$ be the $q$ smallest elements of $A$, and
write $J'=\{r_1<\cdots<r_q\}$.  \eqref{eq:minimum-of-W} exactly the proper condition
 $k_i\leq r_i$.  Apply Lemma~\ref{lem:pascal-cardinal} to $J'$ and $K$, taking $m$ as the
distinguished node, and call the resulting polynomial $F$.  Its Newton
support is contained in $K\subseteq[0,m]\setminus E$, hence
$\Delta^eF(0)=0$ for every $e\in E$.  It vanishes on $J'\setminus\{m\}$.
If $j\in J\setminus J'$, then $j<c\leq k$ for every $k\in K$, so
$\binom jk=0$ and again $F(j)=0$.  Thus all interpolation requirements
hold.

It remains to show such $F$ satisfying \eqref{eq:decoder-inner}, \eqref{eq:decoder-outer} and \eqref{eq:decoder-l1}. Fix $0\leq t\leq m$ and put
$s=q-1$.  Assume there are $a$ elements of $J'\setminus\{m\}$ lie below $m$, and $b$ elements lie
above m, then $a+b=s$.  Their positive integral distances from $m$ are
distinct on each side, and hence
\[
 \prod_{r\in J'\setminus\{m\}}|m-r|\geq a!b!.
\]
Every numerator factor in Lemma~\ref{lem:pascal-cardinal} is at most
$L$.  If $s=0$ there is nothing to prove.  If $s\geq1$, then
\[
 a!b!=\frac{s!}{\binom sa}\geq\frac{s!}{2^s}
 \geq\left(\frac{s}{2\econst}\right)^s,
\]
where the last bound follows from
\[
 \log(s!)=\sum_{j=1}^s\log j
 \geq\int_1^s\log x\,\dd x
 =s\log s-s+1\geq s\log(s/\econst).
\]
Therefore,
\[
 |F(t)|\leq \frac{L^s}{a! b!}\leq \left(\frac{2\econst L}{s}\right)^s\leq(2\econst)^L.
\]
For the last inequality, note that
$x\mapsto x\log(2\econst L/x)$ is increasing on $0<x\leq L$.  This proves
\eqref{eq:decoder-inner}.

Finally, \eqref{eq:finite-difference-inversion} and \eqref{eq:decoder-inner} give, for $0\leq i\leq m$,
\[
 |d_i|=\left|\sum_{u=0}^i(-1)^{i-u}\binom iuF(u)\right|
 \leq2^i(2\econst)^L.
\]
Consequently, again by \eqref{eq:newton-expansion}, for $0\leq t\leq L$,
\[
 |F(t)|\leq(2\econst)^L\sum_{i=0}^{\min(m,t)}2^i\binom ti
 \leq(2\econst)^L3^t\leq(6\econst)^L,
\]
and
\[
 \sum_{i=0}^m|d_i|
 \leq(2\econst)^L\sum_{i=0}^m2^i
 <2^{m+1}(2\econst)^L\leq2(4\econst)^L.
\]
This proves \eqref{eq:decoder-outer} and \eqref{eq:decoder-l1}.
\end{proof}

\subsection{Quantitative Pascal UP}
Next we use the above interpolation result, Theorem \ref{thm:mixed-decoder}, to prove a quantitative version of Pascal uncertainty principle. We first handle the one variable cases.

\begin{lem}[One-variable quantitative Pascal UP]
\label{lem:weighted-two-shift}
Let $0\leq m\leq L$ and let $q(y)=\sum_{j=0}^La_jy^j$ with $a_m\neq0$.
Write $q(y+1)=\sum_{i=0}^Lb_iy^i$ and set
\[
 w_j:=j!(L-j)!,\qquad A_j:=w_ja_j,\qquad B_j:=w_jb_j.
\]
The sequences $(A_j)_{j=0}^L$ and $(B_j)_{j=0}^L$ contain the information of coefficients of $q(y)$ and $q(y+1)$ respectively. Then,
\begin{equation}
	\# \left\{ j: 0\leq j\leq L, |A_j|\geq \delta_L|A_m| \right\}+ \# \left\{ j: 0\leq j\leq L, |B_j|\geq \delta_L|A_m| \right\}\geq m+2.
\end{equation}
Here the decay rate $\delta_L$ is defined by
\begin{equation}
	 \delta_L:=\frac{1}{4(L+1)(12\econst)^L}=\exp \left( -O(L) \right).
\end{equation}
\end{lem}

\begin{proof}
Suppose the conclusion is false and define
\[
 J:=\{j:|A_j|\geq\delta_L|A_m|\},\qquad
 E:=\{i:|B_i|\geq\delta_L|A_m|\}.
\]
Then $m\in J$ and $|J|+|E|\leq m+1$.  Apply
Theorem~\ref{thm:mixed-decoder}, we can obtain a polynomial
\[
 F(x)=\sum_{i=0}^md_i\binom xi
\]
with $F(m)=1$, $F=0$ on $J\setminus\{m\}$,
$d_i=0$ for $i\in E\cap\{0,\ldots,m\}$, and
\begin{equation}\label{eq:F-exp-bound}
	 \max_{0\leq j\leq L}|F(j)|\leq(6\econst)^L,
 \qquad \sum_{i=0}^m|d_i|\leq2(4\econst)^L.
\end{equation}

We extend the Newton coefficient by $d_i:=0$ for $m<i\leq L$ and
then $F(j)=\sum_{i=0}^Ld_i\binom ji$. In particular, $d_i=0$ for every
$i\in E$.
Recalling \eqref{eq:ordinary-Pascal-transform}, we have the following identity
\begin{equation}\label{eq:weighted-duality}
 \sum_{j=0}^La_jF(j)
 =\sum_{j=0}^La_j\sum_{i=0}^jd_i\binom ji
 =\sum_{i=0}^Lb_id_i.
\end{equation}
The interpolation conditions give the two exact reductions
\begin{align}
 \sum_{j=0}^La_jF(j)
 &=a_m+\sum_{j\notin J}a_jF(j),
 \label{eq:duality-left-reduction}\\
 \sum_{i=0}^Lb_id_i
 &=\sum_{i\notin E}b_id_i.
 \label{eq:duality-right-reduction}
\end{align}
Combining
\eqref{eq:weighted-duality}--\eqref{eq:duality-right-reduction} yields
\begin{equation}\label{eq:isolate-am}
 a_m=\sum_{i\notin E}b_id_i-\sum_{j\notin J}a_jF(j).
\end{equation}
For $j\notin J$ and $i\notin E$,
\begin{equation}\label{eq:small-a-b}
 |a_j|<\frac{\delta_L|A_m|}{w_j},
 \qquad
 |b_i|<\frac{\delta_L|A_m|}{w_i}.
\end{equation}
Take absolute values in \eqref{eq:isolate-am}, use
$|a_m|=|A_m|/w_m$, and substitute \eqref{eq:small-a-b}. We obtain that
\begin{equation}\label{eq:weighted-decoder-contradiction-start}
 1
 <\delta_L\left(
   \sum_{j\notin J}\frac{w_m}{w_j}|F(j)|
   +\sum_{i\notin E}\frac{w_m}{w_i}|d_i|\right).
\end{equation}
However, for every $u$,
\begin{equation}\label{eq:weight-ratio}
 \frac{w_m}{w_u}=\frac{\binom Lu}{\binom Lm}\leq2^L.
\end{equation}
Consequently, puting \eqref{eq:F-exp-bound} and \eqref{eq:weight-ratio} into \eqref{eq:weighted-decoder-contradiction-start}, we have
\begin{align}
 1
 &<\delta_L2^L\bigl((L+1)(6\econst)^L+2(4\econst)^L\bigr)\leq\delta_L\,3(L+1)(12\econst)^L
 =\frac34,
 \label{eq:weighted-final-contradiction}
\end{align}
a contradiction.
\end{proof}

Just like \cite[Lemma 5.1]{Li26}, Theorem \ref{lem:weighted-two-shift} can be inductively lifted to multi-variable case.

\begin{lem}[Multi-variable quantitative Pascal UP]\label{lem:weighted-tensor}
Let $p\in\mathbb R_{\leq L} [x_1,\ldots,x_d]$ have total degree at most $L$.  Put
\[
 \Delta_{\leq L}^{(d)}
 :=\{\nu\in\mathbb Z_{\geq0}^d:|\nu|\leq L\}.
\]
Suppose that there is a $\omega\in \Delta_{\leq L}^{(d)} $ such that $[x^\omega]p\neq0$.
For $S\subseteq \{1,2,\cdots ,d\}$ and $\nu\in\Delta_{\leq L}^{(d)}$, denote 
\[
 C_S(\nu):=\nu!\,(L-|\nu|)!\,[x^\nu]p(x+\chi_S).
\]
Then 
\begin{equation}\label{eq:transversality-multi-variable-Pascal-UP}
   \sum_{S\subset \{1,2,\cdots, d\} } \# \left\{ \nu \in\Delta_{\leq L}^{(d)} :  |C_S(\nu)| \geq  \delta_L^d|C_\varnothing(\omega)|  \right\} \geq  \prod_{i=1}^d(\omega_i+2).
\end{equation}	

\end{lem}

\begin{proof}
We induct on $d$. The assertion for $d=1$ is ensured by Lemma~\ref{lem:weighted-two-shift}.  Assume the result holds 
for $d-1$ variables.  Write
\begin{equation}\label{eq:tensor-splitting}
 x=(x',y),\qquad \omega=(\omega',m),
 \qquad
 p(x',y)=\sum_{j=0}^Lp_j(x')y^j.
\end{equation}
The polynomial $p_m$ has total degree at most $L-m$, and the induction hypothesis gives
$\prod_{i<d}(\omega_i+2)$-many distinct pairs $(S',\nu')$ with
\begin{equation}\label{eq:inductive-pairs}
 S' \subseteq\{1,2,\cdots,d-1\},\qquad
 \nu'\in\Delta_{\leq L-m}^{(d-1)}.
\end{equation}
For those $(S',\nu')$, we have 
\begin{align}
\notag \left|\nu'!(L-m-|\nu'|)!
 [x'^{\nu'}]p_m(x'+\chi_{S'})\right|
  &\geq\delta_{L-m}^{d-1}
 \left|\omega'!(L-m-|\omega'|)![x'^{\omega'}]p_m\right|\\ 
 &=\delta_{L-m}^{d-1}
 \left|\omega'!(L-|\omega|)![x^\omega]p\right|.
 \label{eq:inductive-weighted-inequality}
\end{align}
For each pair $(S',\nu')$, we define the fiber polynomial by
\begin{equation}\label{eq:fiber-definition}
 q_{S',\nu'}(y):=[x'^{\nu'}]p(x'+\chi_{S'},y).
\end{equation}
The one can easily check that, for every $j$,
\[
 [y^j]q_{S',\nu'}(y)
 =[x'^{\nu'}]p_j(x'+\chi_{S'}).
\]
Since $\deg p_j\leq L-j$, the right-hand side vanishes whenever
$|\nu'|>L-j$.  Hence every nonzero $y^j$ coefficient satisfies
$j\leq L-|\nu'|$.  Thus the degree of the fiber is at most
\begin{equation}\label{eq:fiber-capacity}
 L':=L-|\nu'|.
\end{equation}
As in Lemma \ref{lem:weighted-two-shift}, the $A_m$ that discribe the $y^m$ coefficient in the fiber polynomial $q_{S',\nu'}(y)$ is
\begin{align}
 A_m^{\rm fib} (S',\nu')
 &:=m!(L'-m)![y^m]q_{S',\nu'}(y)
 \notag\\
 &=m!(L-m-|\nu'|)![x'^{\nu'}]
 p_m(x'+\chi_{S'}).
 \label{eq:fiber-target}
\end{align}
Multiplying \eqref{eq:inductive-weighted-inequality} by $m!$ gives
\begin{equation}\label{eq:fiber-target-lower}
 \nu'!|A_m^{\rm fib} (S',\nu')|
 \geq\delta_{L-m}^{d-1}|C_\varnothing(\omega)|.
\end{equation}
The right-hand side is nonzero.  Hence
$[x'^{\nu'}]p_m(x'+\chi_{S'})\neq0$.  Since
$\deg p_m\leq L-m$, this implies $|\nu'|\leq L-m$, and therefore
$L'=L-|\nu'|\geq m$.  In particular, $A_m^{\rm fib} (S',\nu') \neq0$, so Lemma~\ref{lem:weighted-two-shift} can be applied to the fiber $q_{S',\nu'}(y)$ with degree $L'$.
For each fiber, this produces $m+2$ pairs $(\epsilon,\ell)$, where
$\epsilon\in\{0,1\}$, such that
\begin{equation}\label{eq:fiber-output}
 \ell!(L'-\ell)!
 \left|[y^\ell]q_{S',\nu'}(y+\epsilon)\right|
 \geq\delta_{L'}|A_m^{\rm fib}|.
\end{equation}
Put
\begin{equation}\label{eq:full-output-index}
 S:=\begin{cases}
 S'\cup \{d\},\ \text{ if }\epsilon=1,\\
 S',\  \text{ if }\epsilon=0,\\
 \end{cases} 
 \qquad \nu:=(\nu',\ell).
\end{equation}
Because $0\leq\ell\leq L'=L-|\nu'|$, we have
$\nu\in \Delta_{\leq L}^{(d)}$.
Then, because $L'-\ell=L-|\nu|$,
\begin{align}
 |C_S(\nu)|
 &=\nu'!\ell!(L'-\ell)!
 \left|[x'^{\nu'}y^\ell]
 p(x'+\chi_{S'},y+\epsilon)\right|
 \notag\\
 &=\nu'!\ell!(L'-\ell)!
 \left|[y^\ell]q_{S',\nu'}(y+\epsilon)\right|
 \notag\\
 &\geq\delta_{L'}\nu'!|A_m^{\rm fib}|
 \notag\\
 &\geq\delta_{L'}\delta_{L-m}^{d-1}
 |C_\varnothing(\omega)|.
 \label{eq:tensor-full-output-bound}
\end{align}
The function $N\mapsto\delta_N$ is decreasing, and
$L',L-m\leq L$.  Hence
\begin{equation}\label{eq:delta-monotone-product}
 \delta_{L'}\delta_{L-m}^{d-1}\geq\delta_L^d.
\end{equation}
The construction in \eqref{eq:full-output-index} retains $(S',\nu')$, so outputs
arising from different fibers are distinct.
Therefore the cardinality of $(S,\nu)$ that satisfying \eqref{eq:tensor-full-output-bound} is larger than
\begin{equation}\label{eq:tensor-count-product}
 (m+2)\prod_{i<d}(\omega_i+2)
 =\prod_{i=1}^d(\omega_i+2).
\end{equation}
This proves \eqref{eq:transversality-multi-variable-Pascal-UP}.
\end{proof}

Finally, we polarize the cardinality weight in the summation \eqref{eq:transversality-multi-variable-Pascal-UP} from $S\neq \varnothing,\{1,\cdots,d\}$ to the end point $\varnothing,\{1,\cdots,d\}$. This need to leverage the Hamming distance on the Hamming cube.

\begin{cor}[Polarize to $p(x)$ and $p(x+\one_d)$]\label{cor:polar-Pascal-UP}
Under the same assumptions and notations as in Lemma \ref{lem:weighted-tensor}, we have 
\begin{equation}\label{eq:polar-transversality-multi-variable-Pascal-UP}
   \sum_{S=\varnothing , \{1,2,\cdots, d\} } \# \left\{ \nu \in\Delta_{\leq L}^{(d)} :  |C_S(\nu)| \geq  \frac{\delta_L^d}{2(h+1)^L} |C_\varnothing(\omega)|  \right\} \geq 2^{-d}(L+1)^{-h} \prod_{i=1}^d(\omega_i+2).
\end{equation}	
Here $h=\lfloor d/2 \rfloor$.
\end{cor}
\begin{proof}
	For $S\subset \{1,\cdots,d\}$, we first consider the case that $\# S \leq h$. View $p(x+\chi_S)$ comes from $p(x)=\sum_{\beta \in\Delta^{(d)}_{\leq L}} c_{\beta} x^{\beta}$. We have 
	\begin{align}\label{eq:expand-p-shifted-by-S}
	\notag	p(x+\chi_S)& =\sum_{\beta\in \Delta^{(d)}_{\leq L}}c_{\beta} (x+\chi_S)^{\beta} \\
		  \notag      &= \sum_{\beta\in \Delta^{(d)}_{\leq L}} c_\beta \prod_{i\notin S} x_i^{\beta_i} \prod_{j\in S} (x_j+1)^{\beta_j}\\
						        &= \sum_{\beta\in \Delta^{(d)}_{\leq L}} c_\beta \prod_{i\notin S} x_i^{\beta_i} \prod_{j\in S} \left(\sum_{\alpha_j=0}^{\beta_j} \binom{\beta_j}{\alpha_j} x_j^{\alpha_j}\right).
	\end{align}
    This gives that 
	\[[x^{\alpha}]p(x+\chi_S)= \sum_{\beta\in \mcI(\alpha,S)} [x^{\beta}]p \cdot \prod_{1\leq i \leq d} \binom{\beta_i}{\alpha_i}.\]
	Here,
	\begin{equation}\label{eq:index-admissible-alpha-S}
		\mcI(\alpha,S):= \left\{  \beta\in \Delta^{(d)}_{\leq L} : \beta_i=\alpha_i \ \text{for} \ i\notin S; \ \beta_i\geq \alpha_i \ \text{for} \ i\in S  \right\}
	\end{equation}
	Recall the notation $C_{S}(\nu)$ in Lemma \ref{lem:weighted-tensor}, we have 
	\[[x^{\alpha}]p(x+\chi_S) = \frac{C_S(\alpha)}{\alpha! (L-|\alpha|)!}, \ [x^{\beta}]p(x) = \frac{C_{\varnothing}(\beta)}{\beta! (L-|\beta|)!}.\]
    Hence, 
	\begin{equation}\label{eq:polarize-matrix-to-empty}
			 C_S(\alpha)= \sum_{\beta\in \mcI(\alpha,S)} C_{\varnothing}(\beta) \frac{\alpha! (L-|\alpha|)!}{\beta! (L-|\beta|)!} \cdot \prod_{1\leq i \leq d} \binom{\beta_i}{\alpha_i}.
	\end{equation}
	Since 
	\begin{align*}
		\sum_{\beta\in \mcI(\alpha,S)}  \frac{\alpha! (L-|\alpha|)!}{\beta! (L-|\beta|)!} \cdot \prod_{1\leq i \leq d} \binom{\beta_i}{\alpha_i} &= \sum_{\beta\in \mcI(\alpha,S)}  \frac{\alpha! (L-|\alpha|)!}{\beta! (L-|\beta|)!} \cdot \prod_{1\leq i \leq d} \frac{\beta_i !}{\alpha_i ! (\beta_i-\alpha_i)!}\\
		            &= \sum_{\beta\in \mcI(\alpha,S) \atop \delta=\beta-\alpha} \frac{(L-|\alpha|)!}{(L-|\alpha|-|\delta|)! \delta!}\\
		            &= \sum_{\delta \in \Delta^{(\# S)}_{ \leq L-|\alpha|}} \frac{(L-|\alpha|)!}{(L-|\alpha|-|\delta|)! \delta!}\\
		            &\leq  \sum_{\delta \in \Delta^{(h)}_{ \leq L-|\alpha|}} \frac{(L-|\alpha|)!}{(L-|\alpha|-|\delta|)! \delta!}\\
					&=\sum_{q=0}^{L-|\alpha|}\binom {L-|\alpha|}{q} \sum_{ \delta\in \Delta^{(h)}_{q} } \frac{q!}{\delta!}\\
                   &=\sum_{q=0}^{L-|\alpha|}\binom {L-|\alpha|}{q} h^q
                       =(h+1)^{L-|\alpha|}\leq(h+1)^L.
	\end{align*}
	Here we used multinomial theorem to deduce that 
	\[\sum_{ \delta\in \Delta^{(h)}_{q} } \frac{q!}{\delta!} =h^q.\]
	Therefore, combining the above estimate with \eqref{eq:polarize-matrix-to-empty}, we have 
	\[\sup_{\beta\in \mcI(\alpha,S)} |C_{\varnothing}(\beta)|\geq (h+1)^{-L}|C_S(\alpha)|.\]
	This implies that for each $\alpha\in \left\{ \nu \in\Delta_{\leq L}^{(d)} :  |C_S(\nu)| \geq  \delta_L^d|C_\varnothing(\omega)|  \right\} $, we can find a 
	\[\hat{\beta}(\alpha)\in \mcI(\alpha,S)\cap \left\{ \nu \in\Delta_{\leq L}^{(d)} :  |C_{\varnothing} (\nu)| \geq (h+1)^{-L} \delta_L^d|C_\varnothing(\omega)|  \right\} .\]
	By double counting trick, we have 
	\begin{align*}
		\# &\left\{ \nu \in\Delta_{\leq L}^{(d)}  :  |C_S(\nu)| \geq  \delta_L^d|C_\varnothing(\omega)|  \right\} =\#\{(\alpha,\hat{\beta}(\alpha))\} \\
		      &\leq \sum_{\alpha\in \Delta^{(d)}_{\leq L}} \#\left(\mcI(\alpha,S)\cap \left\{ \nu \in\Delta_{\leq L}^{(d)} :  |C_{\varnothing} (\nu)| \geq (h+1)^{-L} \delta_L^d|C_\varnothing(\omega)|  \right\}   \right)\\
			  & \leq \# \left\{ \nu \in\Delta_{\leq L}^{(d)} :  |C_{\varnothing} (\nu)| \geq (h+1)^{-L} \delta_L^d|C_\varnothing(\omega)|  \right\} \cdot \sup_{\beta\in \Delta^{(d)}_{\leq L}}\# \{\alpha:\beta\in \mcI\alpha \} \\
			  & \leq (L+1)^h\cdot \# \left\{ \nu \in\Delta_{\leq L}^{(d)} :  |C_{\varnothing} (\nu)| \geq (h+1)^{-L} \delta_L^d|C_\varnothing(\omega)|  \right\}.
	\end{align*}
	The last inequality used the fact that in \eqref{eq:expand-p-shifted-by-S}, expanding $(x+\chi_S)^\beta$ can generate 
	\[\prod_{i\in S} (\beta_i+1)\leq (L+1)^h\]
	many terms. Therefore we prove that 
	{\small
	\begin{equation}\label{eq:cardinality-polarize-S-small}
		\# \left\{ \nu \in\Delta_{\leq L}^{(d)}  :  |C_S(\nu)| \geq  \delta_L^d|C_\varnothing(\omega)|  \right\} \leq (L+1)^h\cdot \# \left\{ \nu \in\Delta_{\leq L}^{(d)} :  |C_{\varnothing} (\nu)| \geq (h+1)^{-L} \delta_L^d|C_\varnothing(\omega)|  \right\}.
	\end{equation}
	}
	for all $S$ with $|S|\leq h$. If $|S|>h$, then $|\{1,\cdots,d\}\setminus S|\leq h$.  Viewing $p(x+\chi_S)$ comes from $p'(x):=p(x+\one_d)$ by 
	\[p(x+\chi_S)=p(x+\one_d-\chi_{\{1,\cdots,d\}\setminus S})=p'(x-\chi_{\{1,\cdots,d\}\setminus S}),\]
	same argument will also yield
		{\small
	\begin{equation}\label{eq:cardinality-polarize-S-large}
		\# \left\{ \nu \in\Delta_{\leq L}^{(d)}  :  |C_S(\nu)| \geq  \delta_L^d|C_\varnothing(\omega)|  \right\} \leq (L+1)^h \cdot \# \left\{ \nu \in\Delta_{\leq L}^{(d)} :  |C_{\{1,\cdots,d\}} (\nu)| \geq (h+1)^{-L} \delta_L^d|C_{\varnothing}(\omega)|  \right\}.
	\end{equation}
	}
	Substitute \eqref{eq:cardinality-polarize-S-small} and \eqref{eq:cardinality-polarize-S-large} into \eqref{eq:transversality-multi-variable-Pascal-UP}, and we have 
	\begin{align*}
		  2^d (L+1)^h \sum_{S=\varnothing , \{1,2,\cdots, d\} } \# \left\{ \nu \in\Delta_{\leq L}^{(d)} :  |C_S(\nu)| \geq  \frac{\delta_L^d}{2(h+1)^L} |C_\varnothing(\omega)|  \right\} \geq  \prod_{i=1}^d(\omega_i+2).
	\end{align*}
	The factor $2^d$ is because there are totally $2^d$ many $S\subset \{1,\cdots,d\}$. This proves \eqref{eq:polar-transversality-multi-variable-Pascal-UP}.
\end{proof}

\section{Proof of Theorem \ref{thm:main-anisotropic-exact}}\label{sec:prove-QUC-on-simplex}
In this section, we prove the QUC for the simplex lattice $\Delta_N^{(n)}$ under the exact reccurence relationship \eqref{eq:anisotropic-exact}. The strategy is first estimate the cardinality of transversal point on each shell, and then sum the estimate for all shells. 

As in Section \ref{subsection:factorial-normalization}, we take  
\[
 F(z_1,\ldots,z_n)
 :=\sum_{\alpha\in\Delta_{nR}^{(n)}}
 g(\alpha)\frac{z^\alpha}{\alpha!},
 \]
 and by Lemma \ref{lem:recurrence-D}, \eqref{eq:anisotropic-exact} is equavalent to $DF=0$. We first proves the following lemma, which estimates the transversality on shells for $H\in \ker(D)$: 

\begin{lem}[Shell estimate]
\label{lem:two-facet}
Let $n\ge 3$, $N\ge1$, $d_0=n-2$, and
$h_0=\lfloor d_0/2\rfloor$.  There exist $C_n,c_n>0$ such that the following thing holds: Let $H$ be homogeneous of degree $N$,
$D H=0$, and define
\[
 h_H(\alpha)=\alpha![z^\alpha]H.
\]
Suppose $h_H(a)\ne0$, where $a\in\Delta_N^{(n)}$. Here, we do not assume $a$ satisfying the decreasing condition \eqref{eq:ordered-a}. 
Then 
\begin{equation}
	\#\left\{  \nu \in \partial \Delta^{(n)}_N:  |h_H(\nu)|\ge\exp(-C_nN)|h_H(a)|  \right\} \geq  c_n\frac{\prod_{i=1}^{n-2}(a_i+1)}{(N+1)^{h_0}}
\end{equation}
Here, we denote $\partial \Delta^{(n)}_N:= \{\nu \in \Delta^{(n)}_N: \min_i \nu_i=0\}$.
\end{lem}

\begin{proof}
By Proposition \ref{lem:diagonal-invariance}, if we set
\[
 B(u_1,\ldots,u_{n-1})=H(u_1,\ldots,u_{n-1},0),
\]
then
\begin{equation}\label{eq:facet-reconstruction}
 H(z)=B(z_1-z_n,\ldots,z_{n-1}-z_n).
\end{equation}
Write $B(u)=\sum_{\gamma\in \Delta^{(n-1)}_N }b_\gamma u^\gamma$.  Expanding the
coefficient at $a$ in \eqref{eq:facet-reconstruction} gives
\begin{equation}\label{eq:arbitrary-extraction}
 [z^a]H=(-1)^{a_n}
 \sum_{\substack{|\gamma|=N\\ \gamma_i\ge a_i\ (i<n)}}
 b_\gamma\prod_{i=1}^{n-1}\binom{\gamma_i}{a_i}.
\end{equation}
Put $\theta_i=\gamma_i-a_i$.  Then
$\theta_1+\cdots+\theta_{n-1}=a_n$.  If we denote
$G_\gamma: =\gamma!b_\gamma$, multiplying
\eqref{eq:arbitrary-extraction} by $a!$ yields 
\begin{equation}\label{eq:weighted-arbitrary-extraction}
 h_H(a)=(-1)^{a_n}
 \sum_{|\theta|=a_n}
 \frac{a_n!}{\theta_1!\cdots\theta_{n-1}!}
 G_{(a_1+\theta_1,\ldots,a_{n-1}+\theta_{n-1})}.
\end{equation}
By multinomial theorem,
\[\sum_{|\theta|=a_n}
 \frac{a_n!}{\theta_1!\cdots\theta_{n-1}!}=(n-1)^{a_n}.\]
 Therefore, \eqref{eq:arbitrary-extraction} means that there is some $\gamma\in \Delta^{(n-1)}_N$ satisfies
\begin{equation}\label{eq:deep-arbitrary-coefficient}
 |G_\gamma|\ge(n-1)^{-a_n}|h_H(a)|,
 \qquad \gamma_i\ge a_i\quad(i<n).
\end{equation}
Now take 
\[p(x_1,\ldots,x_{n-2})=B(x_1,\ldots,x_{n-2},1).\]   
Then $p$ is a polynomial (not necessarily homogeneous) with degree at most $N$. For $\omega\in \Z_{\geq 0}^{n-2},|\omega|\le N$, since $B(x_1,\cdots,x_{n-1})$ is homogeneous of degree $N$, the coefficient of $x_1^{\omega_1}\cdots x_{n-2}^{\omega_{n-2}}$ in $p$ is exactly 
the coefficient of $x_1^{\omega_1}\cdots x_{n-2}^{\omega_{n-2}}x_{n-1}^{L-|\omega|}$ in $B$. 
Denote
\begin{align}
 P_\omega
 &=\omega!(N-|\omega|)![x^\omega]p(x),\label{eq:first-weighted-chart}\\
 P_\omega^+
 &=\omega!(N-|\omega|)![x^\omega]p(x+\one_{n-2}).
 \label{eq:second-weighted-chart}
\end{align}
Then first quantity is exactly 
\begin{equation}\label{eq:first-chart-facet-map}
 P_\omega=h_H(\omega_1,\ldots,\omega_{n-2},N-|\omega|,0).
\end{equation}
Moreover, if we define $B'(u_1,\cdots,u_{n-2},u_n)=H(u_1,\cdots,u_{n-2},0,u_{n})$, by \eqref{eq:facet-reconstruction} we have
\[B'(z_1,z_2,\cdots,z_{n-2},z_n)=B(z_1-z_n,\cdots,z_{n-2}-z_n,-z_n).\] 
Setting $z_n=-1$ yields
\[B'(x_1,x_2,\cdots,x_{n-2},-1)=p(x+\one_{n-2}).\]
This means that 
\begin{equation}\label{eq:second-chart-facet-map}
 P_\omega^+=(-1)^{N-|\omega|}
 h_H(\omega_1,\ldots,\omega_{n-2},0,N-|\omega|).
\end{equation}
For $\omega_0=(\gamma_1,\ldots,\gamma_{n-2})$ one has
\begin{equation}\label{eq:P-G-large}
	|P_{\omega_0}|=|G_\gamma|\ge(n-1)^{-a_n}|h_H(a)|. 
\end{equation}
Applying Corollary \eqref{cor:polar-Pascal-UP} to polynomial $p(x)$ with $d_0=n-2, L=N$ gives
{\small
\begin{equation}\label{eq:pascal-compressed-count}
    \# \left\{ \omega \in\Delta_{\leq N}^{(d_0)} :  |P_\omega| \geq  \frac{\delta_N^{d_0}}{2(h_0+1)^N} |P_{\omega_0}|  \right\}  +   \# \left\{ \omega \in\Delta_{\leq N}^{(d_0)} :  |P_\omega^+| \geq  \frac{\delta_N^{d_0}}{2(h_0+1)^N} |P_{\omega_0}|  \right\}  \geq \frac{\prod_{i=1}^{n-2}(\gamma_i+2)}
 {2^{d_0}(N+1)^{h_0}}.
\end{equation}	
}
Here
\begin{equation}\label{eq:pascal-compressed-size}
 \delta_N=\frac{1}{4(N+1)(12 \econst)^N}\sim \exp(-O(N)).
\end{equation}
By \eqref{eq:first-chart-facet-map} and \eqref{eq:second-chart-facet-map}, $|P_{\omega}|$ match with $|h_H|$ on the face $\{\nu\in \partial \Delta^{(n)}_N:\nu_{n}=0\}$, and $|p_{\omega}^+|$ on  $\{\nu\in \partial \Delta^{(n)}_N:\nu_{n-1}=0\}$.
Since an index in their intersection is counted at most twice.
Using \eqref{eq:deep-arbitrary-coefficient}, \eqref{eq:P-G-large}, 
$\gamma_i+2\ge a_i+1$, and taking $C_n$ large enough such that
\[
 \delta_N^{n-2}(h_0+1)^{-N}(n-1)^{-a_n}
 \ge\exp(-C_nN),
\]
we obtain 
\begin{equation}
      \# \left\{ \nu \in \partial \Delta_{ N}^{(n)} :  |h_H(\nu)| \geq  \exp( -C_n N) |h_H(a)|  \right\}  \gtrsim_n \frac{\prod_{i=1}^{n-2}(a_i+1)}{(N+1)^{h_0}}.
\end{equation}	
\end{proof}

\begin{rmk}
  We point out that, \eqref{eq:arbitrary-extraction}  indeed matches with the expansion \cite[(A.19)]{LSZ26} in some sense, which discribe the propagation of the solution of \eqref{eq:anisotropic-exact} form the boundaries of a simplex lattice to its center. Therefore, the obtainment of \eqref{eq:deep-arbitrary-coefficient}, is actually the usage of the so-called ``cone property'' (one can cf. \cite[Section 2.1]{LSZ26}) to the propagation \eqref{eq:arbitrary-extraction}.
\end{rmk}

Now we prove Theorem \ref{thm:main-anisotropic-exact}.

\begin{proof}[Proof of Theorem \ref{thm:main-anisotropic-exact}]
Encode $g$ by
\[
 H(z)=\sum_{|\alpha|=N}g(\alpha)\frac{z^\alpha}{\alpha!}.
\]
Then $DH=0$.  Put $m=a_n$ and,
for $0\le t\le m$, define
\begin{equation}\label{eq:derivative-shell-polynomial}
 H_t=\partial_{z_1}^t\cdots\partial_{z_n}^tH.
\end{equation}
It is homogeneous of degree $N-nt$, belongs to $\ker D$, and one may check that
\begin{equation}\label{eq:derivative-center-preserved}
 (a-t\one_n)![z^{a-t\one_n}]H_t=a![z^a]H=g(a).
\end{equation}

Now Apply Lemma~\ref{lem:two-facet} to $H_t$, we have 
\begin{equation} \label{eq:t-shell-estimate}
	\#\left\{  \nu \in \partial \Delta^{(n)}_{N-nt}:  |h_{H_t}(\nu)|\ge\exp(-C_n(N-nt))|g(a)|  \right\} \geq  c_n\frac{\prod_{i=1}^{n-2}(a_i-t+1)}{(N-nt+1)^{h_0}}.
\end{equation}
Moreover, one can check that 
\[h_{H_t}(\nu)=h_H(\nu+t\one_n)=g(\nu+t\one_n),\quad \forall \nu\in \Delta^{(n)}_{N-nt}.\]
Therefore \eqref{eq:t-shell-estimate} is equavalent to say 
\begin{equation} \label{eq:t-shell-estimate-g}
	\#\left\{  \nu \in \partial^{(t)} \Delta^{(n)}_{N}:  |g(\nu)|\ge\exp(-C_n(N-nt))|g(a)|  \right\} \geq  c_n\frac{\prod_{i=1}^{n-2}(a_i-t+1)}{(N-nt+1)^{h_0}}.
\end{equation}
Here we denote the $t$-shrinked shell by 
\[\partial^{(t)} \Delta^{(n)}_{N} =\{ \nu\in \Delta^{(n)}_{N} : \min_i\nu_i =t\},\]
which are disjoint for different values of $t$. Restrict to $0\le t\le\lfloor m/2\rfloor$, we have
\[
 a_i-t+1\ge\frac12(a_i+1),
 \qquad
 N-nt+1\le N+1.
\]
Summing \eqref{eq:t-shell-estimate-g} over these disjoint shells, we have 
\begin{align*}
		\#\left\{  \nu \in \Delta^{(n)}_{N}:  |g(\nu)|\ge\exp(-C_n N)|g(a)|  \right\} &\geq  c_n \sum_{0\leq t\leq \lfloor m/2\rfloor}\frac{\prod_{i=1}^{n-2}(a_i-t+1)}{(N-nt+1)^{h_0}}\\ 
             & \geq c_n (1+\lfloor m/2\rfloor)\frac{\prod_{i=1}^{n-2}(a_i+1)}{2^{d_0}(N+1)^{h_0}}\\
			 &\gtrsim_n \frac{(a_n+1)\prod_{i=1}^{n-2}(a_i+1)}{(N+1)^{h_0}},
\end{align*}
which proves \eqref{eq:robust-anisotropic-large}. Here, we also used the trivial estimate
\[
 \exp[-C_n(N-nt)]\ge\exp(-C_nN).
\]

\end{proof}

\section{Stability under the approximate recurrence}\label{sec:stability}
This section is to show that, when we perturb the exact reccurence relationship \eqref{eq:anisotropic-exact} to \eqref{eq:robust-simplex-residual}, the QUC result of Theorem \ref{thm:main-anisotropic-exact} is stable.

\subsection{The factorial norm}
We first introduce the factorial norm of homogeneous polynomials.
For a homogeneous polynomial $P(z)$ of degree $m$ written as
\begin{equation}\label{eq:fac-norm-expansion}
 P(z)=\sum_{\alpha\in\Delta_m^{(n)}}
 p(\alpha)\frac{z^\alpha}{\alpha!},
\end{equation}
define its \textbf{$m$-factorial norm} by
\begin{equation}\label{eq:fac-norm-definition}
 \|P\|_{\mathrm{fac},m}
 :=\max_{\alpha\in\Delta_m^{(n)}}|p(\alpha)|
 =\max_{|\alpha|=m}\left|\alpha![z^\alpha]P\right|.
\end{equation}
The following lemma estimate the operator norm of linear change of coordinate under the factorial norm:

\begin{lem}
\label{lem:fac-linear-substitution}
Let $P$ be homogeneous of degree $m$ and let
$A=(a_{ij})_{i,j=1}^n$ be a real matrix. $A$ introduces a linear operator $T_A:\R^{\rm homo}_m[z]\rightarrow \R^{\rm homo}_m[z]$ by 
\[T_A P(z)=P(Az).\]
Set
\begin{equation}\label{eq:column-one-norm}
 \rho(A):=\max_{1\leq j\leq n}\sum_{i=1}^n|a_{ij}|.
\end{equation}
Then
\begin{equation}\label{eq:fac-linear-substitution-bound}
 \|T_A P\|_{\mathrm{fac},m}
 \leq \rho(A)^m\|P\|_{\mathrm{fac},m},\  \forall P\in \R^{\rm homo}_k[z].
\end{equation}
\end{lem}

\begin{proof}
Without loss of generality, we may assume that
$\|P\|_{\mathrm{fac},m}= 1$.  Put
\begin{equation}\label{eq:absolute-linear-forms}
 L_i(z):=\sum_{j=1}^n|a_{ij}|z_j,
 \qquad
 c_j:=\sum_{i=1}^n|a_{ij}|.
\end{equation}
Fix $\alpha\in\Delta_m^{(n)}$, we have 
\begin{align}
 \left|\alpha![z^\alpha]P(Az)\right|
 &\leq
 \alpha![z^\alpha]
 \sum_{|\gamma|=m}\frac{L_1(z)^{\gamma_1}\cdots
 L_n(z)^{\gamma_n}}{\gamma!}
 \notag\\
 &=\frac{\alpha!}{m!}[z^\alpha]
 \left(\sum_{i=1}^nL_i(z)\right)^m
 \notag\\
 &=\frac{\alpha!}{m!}[z^\alpha]
 \left(\sum_{j=1}^nc_jz_j\right)^m
 \notag\\
 &=\prod_{j=1}^nc_j^{\alpha_j}
 \leq\rho(A)^m.
 \label{eq:fac-substitution-row-sum}
\end{align}
For the second line we used the multinomial theorem, and the fourth line follows
because $|\alpha|=m$.  Taking the maximum over $\alpha$ proves the
claim.
\end{proof}

We also need the operator norm estimate for the following multiplier $\ell (z):\R^{\rm homo}_{N-1}[z]\rightarrow \R^{\rm homo}_{N}[z]$:
\begin{equation}\label{eq:S-ell-definition}
 S(z):=z_1+\cdots+z_n,
 \qquad \ell(z):=\frac{S(z)}n.
\end{equation}

\begin{lem}
\label{lem:fac-multiplication-ell}
If $P$ is homogeneous of degree $N-1$, then
\begin{equation}\label{eq:fac-ell-bound}
 \|\ell P\|_{\mathrm{fac},N}
 \leq\frac Nn\|P\|_{\mathrm{fac},N-1}.
\end{equation}
\end{lem}

\begin{proof}
Write
\begin{equation}\label{eq:E-factorial-expansion}
 P(z)=\sum_{|\beta|=N-1}p(\beta)\frac{z^\beta}{\beta!}.
\end{equation}
For $\alpha\in\Delta_N^{(n)}$, direct computation gives
\begin{align}
 \alpha![z^\alpha]\ell(z)P(z)
 &=\frac1n\sum_{i:\alpha_i>0}\alpha_i \cdot p(\alpha-e_i).
 \label{eq:ell-exact-coefficient}
\end{align}
Since $\sum_i\alpha_i=N$, the absolute value of
\eqref{eq:ell-exact-coefficient} is smaller than
$(N/n)\|P\|_{\mathrm{fac},N-1}$.
\end{proof}

\subsection{Decompose polynomials by projection on $\ker D$}

Denote by $\mathcal H_m^{(n)}:=\R^{\rm homo}_{m}[z_1,\cdots,z_n]$ the vector space of homogeneous real
polynomials of degree $m$ in $n$ variables. Let $\mathbf J:=\one_n\one_n^{\mathsf T}$ denote the all-ones
$n\times n$ matrix, and for $0\leq t\leq1$
put
\begin{equation}\label{eq:At-definition}
 A_t:=I_n-\frac tn\mathbf J.
\end{equation}
Thus
\begin{equation}\label{eq:At-action}
 A_tz=z-t\ell(z)\one_n,
 \qquad A_1\one_n=0.
\end{equation}

\begin{lem}[$\ker D$-decomposition]
\label{lem:canonical-diagonal-correction}
Let $N\geq1$, let $F$ be homogeneous of degree $N$, and put
$E:=DF$. Denote
\begin{equation}\label{eq:H-P-definitions}
 H(z):=T_{A_1}F(z)=F(A_1z),
 \qquad P(z):=F(z)-H(z).
\end{equation}
Then
\begin{equation}\label{eq:canonical-properties}
 DH=0,
 \qquad P(z)=S(z)Q(z)
\end{equation}
for a unique homogeneous polynomial $Q$ of degree $N-1$. Especially, this means $T_{A_1}:\mathcal{H}^{(n)}_N\rightarrow \ker D\cap \mcH^{(n)}_N$ acts as a projection on the subspace $\ker D$.  Moreover, 
\begin{align}
 P(z)
 &=\int_0^1\ell(z) DF(A_tz)\,\dd t
 \label{eq:canonical-homotopy}\\
 &=\int_0^{\ell(z)}
 DF \bigl(z-\ell(z)\one_n+u\one_n\bigr)\,\dd u
 \label{eq:canonical-line-integral}\\
 &=\sum_{k=0}^{N-1}
 \frac{(-1)^kS(z)^{k+1}}{n^{k+1}(k+1)!}D^{k+1}F(z).
 \label{eq:canonical-finite-series}
\end{align}
In particular,
\begin{equation}\label{eq:canonical-direct-sum}
 \mathcal H_N^{(n)}
 =\bigl(\ker D\cap\mathcal H_N^{(n)}\bigr)
  \oplus S\mathcal H_{N-1}^{(n)}.
\end{equation}
\end{lem}

\begin{proof}
The chain rule and $A_1\one_n=0$ imply
\begin{equation}\label{eq:DH-zero-chain-rule}
 DH(z)=\langle (\nabla F)(A_1z), A_1\one_n\rangle =0.
\end{equation}
Moreover, if $S(z)=0$, then $A_1z=z$, so $P(z)=0$.  The polynomial $P$ therefore
vanishes on the hyperplane $S=0$ and is divisible by its defining
linear polynomial $S$.  Homogeneity gives $P=SQ$ with $Q$ homogeneous
of degree $N-1$. This proves \eqref{eq:canonical-properties}.

Now differentiate along the homotopy $A_t$ about the parameter $j$:
\begin{align}
 \frac{\dd}{\dd t}T_{A_t}F(z)=\frac{\dd}{\dd t}F(A_tz)
 &=-\ell(z)\sum_{i=1}^n(\partial_{z_i}F)(A_tz)
 =-\ell(z)DF(A_tz).
 \label{eq:homotopy-derivative}
\end{align}
Integrating from $0$ to $1$ proves
\eqref{eq:canonical-homotopy}.  The change of variables
$u=(1-t)\ell(z)$ gives \eqref{eq:canonical-line-integral}.
Do Taylor's expansion about the variable $-v$ gives
\begin{equation}\label{eq:E-diagonal-Taylor}
 DF\bigl(z-v\one_n\bigr)
 =\sum_{k=0}^{N-1}\frac{(-v)^k}{k!}D^{k+1}(z).
\end{equation}
Substituting $v=\ell(z)-u$ in
\eqref{eq:E-diagonal-Taylor} and using
\begin{equation}\label{eq:power-integral}
 \int_0^{\ell(z)}(\ell(z)-u)^k\,\dd u
 =\frac{\ell(z)^{k+1}}{k+1},
\end{equation}
we obtain \eqref{eq:canonical-finite-series}.

To prove the direct sum in \eqref{eq:canonical-direct-sum}, suppose $SQ_0\in\ker D$.  In the coordinates
\begin{equation}\label{eq:barycentric-coordinates}
 u=A_1z,
 \qquad s=\ell(z),
 \qquad z=u+s\one_n,
\end{equation}
the operator $D$ is $\partial_s$.  Thus
$nsQ_0(u+s\one_n)=(SQ_0)(u+s\one_n)$ is independent of $s$.  It
vanishes at $s=0$, hence $nsQ_0(u+s\one_n)\equiv 0$ for any $u$ and $s$.  Therefore
$Q_0=0$.  Every $F\in\mathcal H_N^{(n)}$ has already been written as
$H+SQ$ with $H\in\ker D$.  This proves
\eqref{eq:canonical-direct-sum}.
\end{proof}

Intuitively, in Lemma \ref{lem:canonical-diagonal-correction}, the operator $I-T_{A_1}$ project a polynomial on $(\ker D)^\perp$, whereas $\ker D$ can be viewed as the zero level set of the operator $D$. Hence, it's natural to 
combine the norm of $(I-T_{A_1})F$ and $DF$ together.

\begin{figure}[ht]
\centering
\begin{tikzpicture}[
  x=1cm,
  y=1cm,
  line cap=round,
  line join=round,
  subspace/.style={
    draw=black!75,
    line width=0.9pt
  },
  complement/.style={
    draw=black!70,
    dashed,
    line width=1.0pt
  },
  total/.style={
    draw=red!75!black,
    -{Latex[length=2.8mm]},
    line width=1.35pt
  },
  kernelpart/.style={
    draw=red!65!black,
    -{Latex[length=2.5mm]},
    line width=1.1pt
  },
  correction/.style={
    draw=blue!70!black,
    -{Latex[length=2.5mm]},
    line width=1.1pt
  },
  guide/.style={
    draw=gray!70,
    densely dashed,
    line width=0.75pt
  },
  derivative/.style={
    draw=green!45!black,
    -{Latex[length=2.8mm]},
    line width=1.2pt
  },
  every node/.style={font=\small}
]


\coordinate (O) at (0,0);

\draw[subspace]
  (-4.80,1.056) -- (6.10,-1.342);

\draw[complement]
  (-1.85,-3.33) -- (2.85,5.13);


\coordinate (H) at (4.30,-0.946);
\coordinate (P) at (0.95,1.71);
\coordinate (F) at (5.25,0.764);

\draw[guide] (H) -- (F);
\draw[guide] (P) -- (F);

\draw[kernelpart] (O) -- (H);

\draw[correction] (O) -- (P);

\draw[total] (O) -- (F);

\fill (O) circle (1.5pt);
\node[anchor=north east] at (-0.08,0.07) {$0$};

\node[
  red!65!black,
  anchor=north
] at (4.15,-1.06)
  {$T_{A_1}F$};

\node[
  blue!70!black,
  anchor=east
] at (0.77,1.72)
  {$(I-T_{A_1})F$};

\node[
  red!75!black,
  anchor=west
] at (5.34,0.80)
  {$F$};


\node[
  anchor=west,
  align=left
] at (5.28,-1.82)
  {$\ker D\cap\mathcal H_N^{(n)}$\\[-1mm]
   \scriptsize zero level set of $DF$};

\node[
  anchor=west
] at (2.82,4.64)
  {$S\mathcal H_{N-1}^{(n)}$};


\coordinate (B) at (-2.65,0.583);

\draw[derivative]
  (B) -- (-1.28,4.38);

\node[
  green!45!black,
  anchor=south
] at (-1.38,4.48)
  {direction detected by $DF$};

\fill[red!65!black]  (H) circle (1.5pt);
\fill[blue!70!black] (P) circle (1.5pt);

\end{tikzpicture}

\caption{
Geometric interpretation of decomposition
\(
F=T_{A_1}F+(I-T_{A_1})F
\) and Theorem \ref{thm:canonical-correction-norm}.
}
\label{fig:canonical-correction-geometry}
\end{figure}
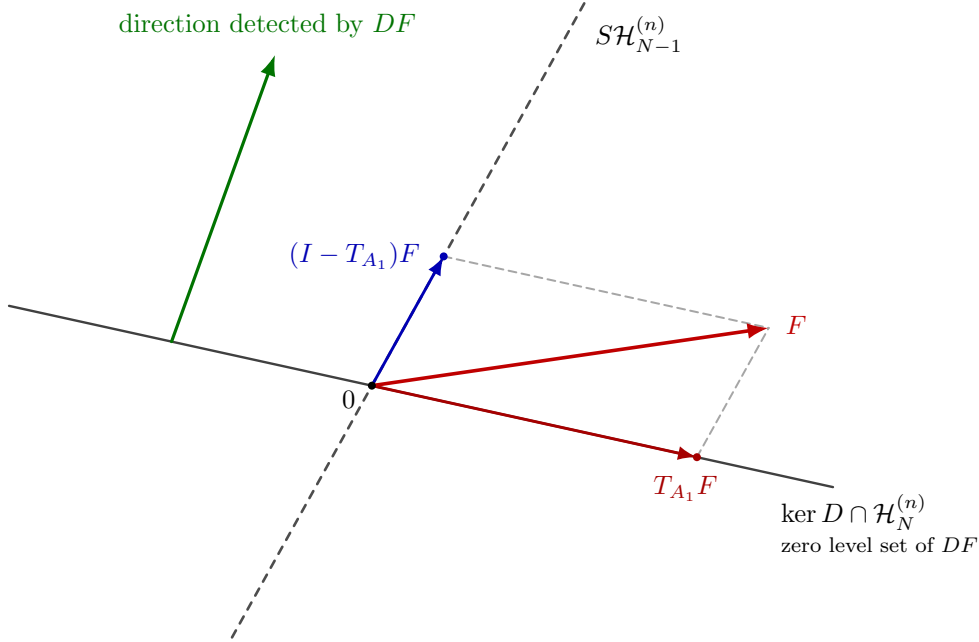

For $E\in\mathcal H_{N-1}^{(n)}$, modeling after \eqref{eq:canonical-homotopy} and \eqref{eq:canonical-line-integral}, we define the operator $\mathcal C_{n,N}:\mcH^{(n)}_{N-1}\rightarrow \mcH^{(n)}_N$ by
\begin{equation}\label{eq:correction-operator-definition}
 (\mathcal C_{n,N}E)(z)
 :=\int_0^1\ell(z) E(A_tz)\,\dd t=\int_0^{\ell(z)}
 E\bigl(z-\ell(z)\one_n+u\one_n\bigr)\,\dd u.
\end{equation}
Doing change of coordinate by $z=u_0+s\one_n$, where $u_0=A_1z$ and $s=\ell(z)$, gives
\begin{equation}\label{eq:correction-barycentric-coordinates}
 (\mathcal C_{n,N}E)(u_0+s\one_n)
 =\int_0^sE(u_0+v\one_n)\,\dd v.
\end{equation}
Since under the coordinate $(u,s)$, we have $D=\partial_s$, direct computation yields
\begin{equation}\label{eq:correction-right-inverse}
 D\mathcal C_{n,N}E=E,
 \qquad
 (\mathcal C_{n,N}E)(A_1z)=0.
\end{equation}
Thus $D:\mathcal H_N^{(n)}\to\mathcal H_{N-1}^{(n)}$ is surjective,
and Lemma~\ref{lem:canonical-diagonal-correction} states that
$P=(I-T_{A_1})F=\mathcal C_{n,N}(DF)$.

\begin{thm}
\label{thm:canonical-correction-norm}
Let $n\geq3$, let $F$ be homogeneous of degree $N$, we have the following estimate about the operator norm of $\mathcal C_{n,N}$:
\begin{equation}\label{eq:correction-exact-operator-norm}
 \|\mathcal C_{n,N}\|_{
 (\mathrm{fac},N-1)\to(\mathrm{fac},N)}= \frac{\left(2-\frac2n\right)^N-1}{n-2}:=K_{n,N}.
\end{equation}
In particular, let $P$ be as in Lemma~\ref{lem:canonical-diagonal-correction}.  Then
\begin{equation}\label{eq:canonical-norm-bound}
 \|P\|_{\mathrm{fac},N}
 \leq K_{n,N}\|DF\|_{\mathrm{fac},N-1}.
\end{equation}

\end{thm}

\begin{proof}
Every column of $A_t$ has one diagonal entry $1-t/n$ and $n-1$
off-diagonal entries $-t/n$.  Consequently,
\begin{equation}\label{eq:At-column-norm}
 \rho(A_t)=1-\frac tn+(n-1)\frac tn
 =1+\frac{n-2}{n}t.
\end{equation}
Assuming $E(z)\in \mcH^{(n)}_{N-1}$, Lemma
\ref{lem:fac-linear-substitution} and Lemma
\ref{lem:fac-multiplication-ell} imply
\begin{align}
 \|\ell(z)E(A_tz)\|_{\mathrm{fac},N}
 &\leq\frac Nn
 \left(1+\frac{n-2}{n}t\right)^{N-1}
 \|E\|_{\mathrm{fac},N-1}.
 \label{eq:homotopy-integrand-bound}
\end{align}
Integrating \eqref{eq:homotopy-integrand-bound} and using
\eqref{eq:correction-operator-definition},
\begin{align}
 \|\mathcal C_{n,N} E\|_{\mathrm{fac},N}
 &\leq\frac Nn\int_0^1
 \left(1+\frac{n-2}{n}t\right)^{N-1}\dd t
 \|E\|_{\mathrm{fac},N-1}
 \notag\\
 &=\frac{(2-2/n)^N-1}{n-2}
 \|E\|_{\mathrm{fac},N-1}.
 \label{eq:K-integral-evaluation}
\end{align}
It remains to prove that the equality in \eqref{eq:K-integral-evaluation} can be taken. Set
\begin{equation}\label{eq:extremal-residual-array}
 e(\beta):=(-1)^{\beta_1+\cdots+\beta_{n-1}},
 \qquad \beta\in\Delta_{N-1}^{(n)},
\end{equation}
and let
\begin{equation}\label{eq:extremal-E-polynomial}
 E(z):=\sum_{|\beta|=N-1}e(\beta)\frac{z^\beta}{\beta!}.
\end{equation}
Then $\|E\|_{\mathrm{fac},N-1}=1$.  Since
\begin{equation}\label{eq:canonical-v-integral}
 \mathcal C_{n,N}E (z)=\ell(z)\int_0^1
 E\bigl(z-(1-v)\ell(z)\one_n\bigr)\,\dd v.
\end{equation}
At $z=e_n$, the first $n-1$ coordinate of $z-(1-v)\ell(z)\one_n $ equal
$-(1-v)/n$, while the last equals $(n-1+v)/n$.  The sign in
\eqref{eq:extremal-residual-array} cancels all negative signs of the first $n-1$ coordinates.  Therefore, the
multinomial theorem gives
\begin{align*}
 E\left(-\frac{1-v}{n},\ldots,-\frac{1-v}{n},
          \frac{n-1+v}{n}\right)
 &=\frac1{(N-1)!}
 \left(\frac{2(n-1)-(n-2)v}{n}\right)^{N-1}
\end{align*}
and
\begin{equation} \label{eq:extremal-E-evaluation}
	C_{n,N}E (e_n)= \frac{1}{n \cdot (N-1)!}\int_0^1 \left(\frac{2(n-1)-(n-2)v}{n}\right)^{N-1}\dd v.
\end{equation}
Since $C_{n,N}E$ is homogeneous of degree $N$,
$N![z_n^N]C_{n,N}E=N!C_{n,N}E(e_n)$. Hence by \eqref{eq:extremal-E-evaluation},
\begin{align}
 N!P(e_n)
 &=\frac Nn\int_0^1
 \left(\frac{2(n-1)-(n-2)v}{n}\right)^{N-1}\dd v
 =K_{n,N}.
 \label{eq:extremal-attains-K}
\end{align}
Thus equality is attained in \eqref{eq:canonical-norm-bound}.
\end{proof}

\subsection{Proof of Theorem \ref{thm:robust-anisotropic}}
Now we prove the QUC on simplex lattice with approximate reccurence relationship \eqref{eq:robust-simplex-residual}.
\begin{proof}[Proof of Theorem \ref{thm:robust-anisotropic}]
Decode $g$ to the homogeneous polynomial
\begin{equation}\label{eq:robust-F-definition}
 F(z):=\sum_{\alpha\in\Delta_N^{(n)}}
 g(\alpha)\frac{z^\alpha}{\alpha!}.
\end{equation}

By Lemma \label{lem:recurrence-D}, assumption \eqref{eq:robust-simplex-residual} is exactly
\begin{equation}\label{eq:DF-norm-residual}
 \|DF\|_{\mathrm{fac},N-1}\leq e^{-C_n N}|g(a)|.
\end{equation}
As in the $D$-decomposition Lemma \ref{lem:canonical-diagonal-correction}, we set $H=F(A_1z)$ and $P=F-H$.  Write
\begin{equation}\label{eq:h-p-arrays}
 h(\alpha):=\alpha![z^\alpha]H,
 \qquad
 p(\alpha):=\alpha![z^\alpha]P.
\end{equation}
THen we have $h(\alpha)+p(\alpha)=g(\alpha)$. By Lemma~\ref{lem:canonical-diagonal-correction}, $DH=0$, so the function
$h$ satisfies the exact recurrence \eqref{eq:anisotropic-exact}. Moreover, Theorem
\ref{thm:canonical-correction-norm} and
\eqref{eq:DF-norm-residual} gives
\begin{equation}\label{eq:p-uniform-bound}
 \sup_{\alpha\in \Delta^{(n)}_N}|p(\alpha)|=\|P\|_{\mathrm{fac},N}
 \leq K_{n,N}  e^{-C_n N} |g(a)|.
\end{equation}
Here, 
\[
  K_{n,N}=\frac{\left(2-\frac2n\right)^N-1}{n-2}\sim \exp(O_n(N)).
\]
Therefore, by choosing $C_n\gg 1$, we can ensure that
\begin{equation}\label{eq:h-center-lower}
 |h(a)|\geq |g(a)|-|p(a)|\geq (1-K_{n,N}  e^{-C_n N}) |g(a)|\geq \frac{1}{2} |g(a)|.
\end{equation}
Now we Apply Theorem \ref{thm:main-anisotropic-exact} to $h$. Denote the feasible constant in Theorem \ref{thm:main-anisotropic-exact} by $C'_n$ and $c'_n$. We have 
\begin{equation}\label{eq:h-transversality}
 \#\left\{\alpha\in\Delta_{N}^{(n)}:
 |h(\alpha)|\geq \econst^{-C'_n N}|h(a)| \geq \frac{1}{2}\econst^{-C'_n N} |g(a)|\right\}
 \geq c'_n\Phi_n(a),
\end{equation}
For every $\alpha$ in the left hand side of \eqref{eq:h-transversality}, $g=h+p$ and
\eqref{eq:p-uniform-bound} imply
\begin{align}
 |g(\alpha)|
 &\geq |h(\alpha)|-|p(\alpha)| \geq\left( \frac{1}{2}\econst^{-C'_n N} - K_{n,N}  e^{-C_n N}  \right)|g(a)|.
 \label{eq:robust-transfer-final}
\end{align}
By taking $C_n\gg C'_n$ sufficiently large, and $c_n=c'_n$, we have 
\[
\frac{1}{2}\econst^{-C'_n N} - K_{n,N}  e^{-C_n N} \geq \frac{1}{3}\econst^{-C'_n N}\geq \econst^{-C_n N},
\]
and therefore
\begin{equation}
 \#\left\{\alpha\in\Delta_{N}^{(n)}:
 |g(\alpha)|\geq \econst^{-C_n N}|g(a)| \right\}
 \geq c_n\Phi_n(a),
\end{equation}
This proves \eqref{eq:robust-anisotropic-large}. 
\end{proof}

\section{Proof of Theorem \ref{thm:main}}\label{sec:prove-QUC-on-lattice}
In this section, we prove the QUC for stationary Schr\"odinger equation on $\Z^d$. The strategy is using the QUC for simplex lattice, Theorem \ref{thm:robust-anisotropic}, together with some geometric pigeonholing argument.

Denote 
\[ |x|_1=|x_1|+|x_2|+\cdots+|x_n|\]
the $\ell^1$-norm on $\Z^n$. Let 
\[
 \mathbb S_N (y)=\{x\in\Z^d:|x-y|_1=N\}
\]
be the $\ell^1$-sphere in $\Z^n$ with center $y\in \Z^d$ and radius $N$. In particular, we denote $\mathbb S_N=\mathbb S_N(0)$.

\subsection{The Walsh decomposition of the $\ell^1$-sphere}
We first introduce the Walsh decomposition of the $\ell^1$-sphere $\mathbb S_r$. The key idea is that, since the intersection of $\mathbb S_r$ with each coordinate orthant of the $n$-dimensional space is exactly a simplex lattice $\Delta^{(n)}_r$, we can correspond each orthant (totally $2^n$ many) to a multi-sign $\sigma\in \{\pm 1\}^n$, and (roughly speaking) decompose 
\[\mathbb S_r \approx \{\pm\}^n \times \Delta^{(n)}_{r}.\]
This converts some geometric structure (which is universal for all dimension) of $\Z^n$ to some algebraic structure, and enables us to study functions on $\mathbb S_r$ with algebraic techniques. From this pespective, this ingredient is an affordable alternative to the special geometric structure of $\Z^3$ that applied in \cite{LZ22}.

Let $\mathcal V$ be the vector space with basis $\{e_r\}_{r\in\Z}$. We
graded each $e_r$ by degree $\deg e_r=|r|$.  Define the lowering operator
\begin{equation}\label{eq:one-arm-lowering}
 Le_r=
 \begin{cases}
 e_{r-1},&r>0,\\
 e_{r+1},&r<0,\\
 0,&r=0.
 \end{cases}
\end{equation}
For each $r\geq 0$, we introduce the symmetric and antisymmetric vectors by
\begin{equation}\label{eq:two-arm-basis}
 s_r=e_r+e_{-r},
 \qquad
 c_r=e_r-e_{-r}.
\end{equation}
By our definition, we have $s_0=2e_0$ and $c_0=0$. Moreover, $\deg s_r=\deg c_r=r$. Simple computation shows that 
\begin{equation}\label{eq:two-arm-lowering}
 Ls_r=s_{r-1}\neq 0\ (r\ge1),
 \qquad
 Lc_r=c_{r-1}\neq 0\ (r\ge2),\qquad L s_0=L c_0=Lc_1=0,
 \end{equation}
On the tensor space $\mathcal V^{\otimes n}$, we define
\begin{equation}\label{eq:tensor-lowering}
 \mathscr D=L_1+\cdots+L_n.
\end{equation}
Each $L_j,1\leq j\leq n$ act on the $j$th vector space $\mathcal{V}$.
A basis $\{v_{J,b}\}$ (called the \textbf{Walsh basis}) of $\mathcal V^{\otimes n}$ is indexed by $J\subset\{1,\cdots,n\}$ and $b\in\Z_{\ge0}^n$, with the representation
\begin{equation}\label{eq:sector-basis}
 v_{J,b}=\bigotimes_{i=1}^n w_i,
 \qquad
 w_i=
 \begin{cases}
 c_{b_i+1},&i\in J,\\
 s_{b_i},&i\notin J.
 \end{cases}
\end{equation}
Here we use the index $b_i+1$ since $c_r\neq 0$ only when $r\geq 1$. The degree of $v_{J,b}$ is given by 
\[\deg v_{J,b}=\sum_{i=1}^n \deg w_i = |b|+|J|,\]
and direct computation shows that 
\begin{equation}\label{eq:sector-simplex-action}
 \mathscr Dv_{J,b}=\sum_{i:b_i>0}v_{J,b-e_i}.
\end{equation}
Notice that, if we restrict the degree of a tensor vector to be $N$, then $|J|+|b|=N$, which means that for every fixed $J$ the feasible $b$ lies in the simplex lattice $\Delta^{(n)}_{N-|J|}$. This, indeed, decomposes the $\ell^1$-sphere into sign part $\sigma\in \{\pm 1\}^{n}$ (with $\sigma_j=-1$ if and only if $j\in J$) together with a simplex lattice fiber. We give more precise discription below.

Let $f$ be a function on the $\ell^1$-sphere
\[
 \mathbb S_N=\{x\in\Z^n:|x|_1=N\},
\]
set
\begin{equation}\label{eq:sphere-vector}
 \mathbf f=\sum_{|x|_1=N}f(x)e_{x_1}\otimes\cdots\otimes e_{x_n}.
\end{equation}
Then $f$ is of degree $N$. Under the Walsh basis, write
\begin{equation}\label{eq:sector-expansion}
 \mathbf f=\sum_{J\subset\{1,\cdots,n\}}
 \sum_{b\in\Delta_{N-|J|}^{(n)}}g_J(b)v_{J,b},
\end{equation}
where sectors with $|J|>N$ are absent.

\begin{thm}[Walsh transform]\label{thm:Walsh-transform}
	Let $a=(a_1,\ldots,a_n)\in\Delta^{(n)}_N$, let $I(a)=\{i:a_i>0\}$, and let $\sigma\in\{\pm1\}^{I(a)}$. So $\sigma a\in \mathbb S_N$. For each $J\subset I(a)$, define 
	\[
	 \sigma_J=\prod_{i\in J}\sigma_i.
	\]
	Then we have the Walsh transform 
	\begin{equation}\label{eq:Walsh-trans}
 g_J(a-\one_J)=2^{-n}
 \sum_{\sigma \in \{\pm 1\}^ {I(a)} }\sigma_J f(\sigma a),
\end{equation}
and the reversed Walsh transform
\begin{equation}\label{eq:revers-Walsh-trans}
 f(\sigma a)=2^{n-|I(a)|}
 \sum_{J\subset I(a)}\sigma_Jg_J(a-\one_J).
\end{equation}
Here $\one_J(i)=1$ for $i\in J$, and $\one_J(0)=0$ for $i\notin J$.
\end{thm}

\begin{proof}
We divide the proof into several steps.

\medskip
\noindent
\textbf{Step 1: identify the indices $(J,b)$ that match with $a\in \Delta^{(n)}_N$.}

Fix an $a\in\Delta_N^{(n)}$. Define the corresponding sign-fibre
subspace
\[
\mathcal W_a
:=
\operatorname{span}
\left\{
E_\sigma(a):
\sigma\in\{\pm1\}^{I(a)}
\right\},
\]
where
\begin{equation}\label{eq:physical-sign-vector}
	E_\sigma(a)
:=
e_{(\sigma a)_1}\otimes\cdots\otimes e_{(\sigma a)_d}.
\end{equation}
We now identify that, for which $(J,b)$, the Walsh basis $v_{J,b}$ comes from $\mcW_a$. If $i\notin J$, then the $i$-th tensor factor is
\[
s_{b_i}=e_{b_i}+e_{-b_i}.
\]
So every $e_r$ appearing in $s_{b_i}$ has absolute index $b_i$. If $i\in J$, then the $i$-th tensor factor is
\[
c_{b_i+1}=e_{b_i+1}-e_{-b_i-1},
\]
and every $e_r$ appearing in $c_{b_i+1}$ has absolute index $b_i+1$. Therefore every tensor basis $e_{r_1}\otimes\cdots \otimes e_{r_n}$ appearing in $v_{J,b}$ has
absolute multi-index
\[
b+\mathbf 1_J.
\tag{7.14}
\label{eq:sector-absolute-coordinate}
\]
It follows that
\begin{equation}
	v_{J,b}\in\mathcal W_a \ \Leftrightarrow \ b+\mathbf 1_J=a  \ \Leftrightarrow \ b=a-\mathbf 1_J.
\end{equation}
For $a-\mathbf 1_J$ to be nonnegative coordinatewise, it is necessary
and sufficient that
\[
J\subseteq I(a).
\]
Moreover,
\[
|a-\mathbf 1_J|
=
N-|J|,
\]
so
\[
b= a-\mathbf 1_J\in\Delta_{N-|J|}^{(n)}.
\]
Thus the component of \eqref{eq:sphere-vector} lying in
$\mathcal W_a$ is exactly
\[
\sum_{\sigma\in\{\pm1\}^{I(a)}}
f(\sigma a)E_\sigma(a)
=
\sum_{J\subseteq I(a)}
g_J(a-\mathbf 1_J)
v_{J,a-\mathbf 1_J}.
\tag{7.16}
\label{eq:fibre-sector-expansion}
\]

\medskip
\noindent
\textbf{Step 2: expand $v_{J,a-\mathbf 1_J}$ in the original basis.}

Fix $J\subseteq I(a)$. In the $i$-th coordinate there are three possibilities.
\begin{itemize}
	\item If $i\in J$, then $a_i\geq1$ and
                  \begin{equation}\label{eq:odd-coordinate-expansion}
				  w_i(J,a-\mathbf 1_J)=c_{a_i}=e_{a_i}-e_{-a_i}=
               \sum_{\sigma_i\in\{\pm1\}}
               \sigma_i e_{\sigma_i a_i}.
				  \end{equation} 
	\item If $i\in I(a)\setminus J$, then $a_i\geq1$ and
            \begin{equation}\label{eq:even-coordinate-expansion}
	w_i(J,a-\mathbf 1_J)=s_{a_i}=e_{a_i}+e_{-a_i}
                        =
               \sum_{\sigma_i\in\{\pm1\}}
             e_{\sigma_i a_i}.
             \end{equation}
	\item Finally, if $i\notin I(a)$, then $a_i=0$. Necessarily $i\notin J$,
and
\begin{equation}\label{eq:zero-coordinate-expansion}
w_i(J,a-\mathbf 1_J)
=
s_0
=
2e_0.	
\end{equation}

\end{itemize}
Let
\[
m:=|I(a)|.
\]
Then there are exactly $|\{1,\cdots,n\}\setminus I(a)|=n-m$ zero coordinates of $a$. Hence the tensor product of
the $2e_0$ in \eqref{eq:zero-coordinate-expansion} contributes the
scalar factor $2^{n-m}$.
Tensoring \eqref{eq:odd-coordinate-expansion},
\eqref{eq:even-coordinate-expansion}, and
\eqref{eq:zero-coordinate-expansion}, we obtain
\begin{equation}\label{eq:sector-vector-sign-expansion}
	v_{J,a-\mathbf 1_J}
=
2^{n-m}
\sum_{\sigma\in\{\pm1\}^{I(a)}}
\sigma_J E_\sigma(a).
\end{equation}

Indeed, each coordinate $i\in J$ contributes the sign $\sigma_i$,
whereas every coordinate in $I(a)\setminus J$ contributes the scalar
$1$. Their product is precisely
\[
\prod_{i\in J}\sigma_i=\sigma_J.
\]

\medskip
\noindent
\textbf{Step 3: derive the inverse Walsh transform.}

Substitute \eqref{eq:sector-vector-sign-expansion} into
\eqref{eq:fibre-sector-expansion}. We get
\begin{align*}
\sum_{\sigma\in\{\pm1\}^{I(a)}}
f(\sigma a)E_\sigma(a)
&=
\sum_{J\subseteq I(a)}
g_J(a-\mathbf 1_J)
\left(
2^{n-|I(a)|}
\sum_{\sigma\in\{\pm1\}^{I(a)}}
\sigma_JE_\sigma(a)
\right)\\
&=
2^{n-|I(a)|}
\sum_{\sigma\in\{\pm1\}^{I(a)}}
\left(
\sum_{J\subseteq I(a)}
\sigma_Jg_J(a-\mathbf 1_J)
\right)
E_\sigma(a).
\end{align*}
The vectors
\[
\{E_\sigma(a):\sigma\in\{\pm1\}^{I(a)}\}
\]
are distinct members of the original tensor-product basis and are
therefore linearly independent. Comparing the coefficient of
$E_\sigma(a)$ on the two sides gives
\[
f(\sigma a)
=
2^{n-|I(a)|}
\sum_{J\subseteq I(a)}
\sigma_Jg_J(a-\mathbf 1_J),
\]
which is exactly \eqref{eq:revers-Walsh-trans}.

\medskip
\noindent
\textbf{Step 4: derive the Walsh transform.}

Fix arbitrary $K\subseteq I(a)$ and still denote $m=|I(a)|$. Multiply
\eqref{eq:revers-Walsh-trans} by $\sigma_K$ and sum over all
$\sigma\in\{\pm1\}^{I(a)}$. This gives
\begin{align*}
\sum_{\sigma\in\{\pm1\}^{I(a)}}
\sigma_K f(\sigma a)
&=
2^{n-m}
\sum_{\sigma\in\{\pm1\}^{I(a)}}
\sigma_K
\sum_{J\subseteq I(a)}
\sigma_Jg_J(a-\mathbf 1_J)\\
&=
2^{n-m}
\sum_{J\subseteq I(a)}
g_J(a-\mathbf 1_J)
\sum_{\sigma\in\{\pm1\}^{I(a)}}
\sigma_K\sigma_J.
\tag{7.21}
\label{eq:forward-before-orthogonality}
\end{align*}
One may easily check the following orthogonality of Walsh character $\sigma_J$:
\begin{equation}\label{eq:walsh-character-orthogonality}
	\sum_{\sigma\in\{\pm1\}^{I(a)}}\sigma_K\sigma_J=2^m\mathbf 1_{\{J=K\}}.
\end{equation}
Using \eqref{eq:walsh-character-orthogonality} in
\eqref{eq:forward-before-orthogonality}, we obtain
\begin{align*}
\sum_{\sigma\in\{\pm1\}^{I(a)}}
\sigma_K f(\sigma a)
&=
2^{n-m}2^m g_K(a-\mathbf 1_K)=
2^n g_K(a-\mathbf 1_K).
\end{align*}
Therefore
\[
g_K(a-\mathbf 1_K)
=
2^{-n}
\sum_{\sigma\in\{\pm1\}^{I(a)}}
\sigma_Kf(\sigma a),
\]
which is precisely \eqref{eq:Walsh-trans}. 
\end{proof}

\subsection{Transversality on the $\ell^1$-sphere}
We first convert the reccurence relationship on $\ell^1$-sphere $\mathbb S_N$ to the reccurence relationship on simplex lattice $\Delta^{(n)}_N$. Let $f:\mathbb S_N\rightarrow \R$, recall that \eqref{eq:sphere-vector} encodes $f$ with a tensor vector in $\mathcal{V}^{\otimes n}$ by 
\begin{equation*}
 \mathbf f=\sum_{|x|_1=N}f(x)e_{x_1}\otimes\cdots\otimes e_{x_n}.
\end{equation*}
Recall the definition of $\mathscr{D}$ in \eqref{eq:tensor-lowering}. 
\begin{lem}\label{lem:mathscrD-equivalence}
	The following statements are equavalent:
	\begin{itemize}
		\item[(1)] $\sum_{x\in \mathbb S_N :|x-y|_1=1}f(x)=0$ for all $y\in \mathbb S_{N-1}$;
		\item[(2)] $\mathscr{D}\mathbf f=0$;
		\item[(3)] For each $J\subset \{1,\cdots,n\}$, the function $g_{J}:\Delta^{(n)}_{N-|J|}\rightarrow \R$ satisfies 
		   \[\sum_{i=1}^n g_J(\beta+e_i)=0,\quad \forall \beta\in \Delta^{(n)}_{N-|J|-1}.\] 
	\end{itemize} 
\end{lem}
\begin{proof}
For $x=(x_1,\ldots,x_n)\in\mathbb Z^n$, write
\[
E_x:=e_{x_1}\otimes\cdots\otimes e_{x_n}.
\]
Thus the tensor encoding of $f$ is
\[
\mathbf f=\sum_{x\in S_N}f(x)E_x.
\]
We prove the equivalences in two steps.

\medskip
\noindent
\textbf{Step 1: equivalence of \emph{(1)} and \emph{(2)}.}

Let $x\in S_N$. If $x_i\neq0$, define
\[
x^{(i)}
:=
x-\operatorname{sgn}(x_i)e_i.
\]
Then $x^{(i)}\in \mathbb S_{N-1}$ and $|x-x^{(i)}|_1=1$.
By the definition of $L_i$,
\[
L_iE_x=
\begin{cases}
E_{x^{(i)}},&x_i\neq0,\\
0,&x_i=0.
\end{cases}
\]
Consequently,
\begin{equation}\label{eq:D-on-physical-basis}
	\mathscr DE_x
=
\sum_{\substack{1\leq i\leq n\\x_i\neq0}}
E_{x-\operatorname{sgn}(x_i)e_i}.
\end{equation}
Applying \eqref{eq:D-on-physical-basis} to compute $\mathscr{D}\mathbf f$ by
\begin{align*}\label{eq:Df-first-expansion}
\mathscr D\mathbf f
=
\sum_{x\in S_N}f(x)\mathscr DE_x& =
\sum_{x\in S_N}f(x)
\sum_{\substack{1\leq i\leq n\\x_i\neq0}}
E_{x-\operatorname{sgn}(x_i)e_i}\\
&=
\sum_{y\in S_{N-1}}
\left(
\sum_{\substack{x\in S_N\\|x-y|_1=1}}f(x)
\right)E_y.
\end{align*}
Therefore, $\mathscr D\mathbf f=0$ holds if and only if
\[
\sum_{\substack{x\in S_N\\|x-y|_1=1}}f(x)=0
\]
for every $y\in S_{N-1}$. This proves the equivalence of
\emph{(1)} and \emph{(2)}.

\medskip
\noindent
\textbf{Step 2: equivalence of \emph{(2)} and \emph{(3)}.}
Recalling \eqref{eq:sector-simplex-action}, we have
\[
\mathscr Dv_{J,b}
=
\sum_{\substack{1\leq i\leq n\\b_i>0}}
v_{J,b-e_i}.
\tag{7.20}
\label{eq:D-sector-basis}
\]
In particular, when $|J|=N$, since $b\in \Delta^{(n)}_{N-|J|}$, $b$ must be zero, in which case $v_{J,b}=c_1\otimes c_1\otimes\cdots\otimes c_1$ and $\mathscr{D}v_{J,b}=0$. Recalling the representation of $\mathbf f$ under Walsh basis is \eqref{eq:sector-expansion}, we have 
\begin{align*}\label{eq:Df-sector-unreindexed}
\mathscr D\mathbf f
&=
\sum_{\substack{J\subseteq[n]\\|J|\leq N}}
\ \sum_{b\in\Delta_{N-|J|}^{(n)}}
g_J(b)\mathscr Dv_{J,b}\\
&=
\sum_{\substack{J\subseteq[n]\\|J|\leq N-1}}
\ \sum_{b\in\Delta_{N-|J|}^{(n)}}
g_J(b)
\sum_{\substack{1\leq i\leq n\\b_i>0}}
v_{J,b-e_i}\\
&=
\sum_{\substack{J\subseteq[n]\\|J|\leq N-1}}
\ \sum_{\beta\in\Delta_{N-|J|-1}^{(n)}}
\left(
\sum_{i=1}^n g_J(\beta+e_i)
\right)v_{J,\beta}.
\end{align*}
Here, we used $[n]$ to denote the set $\{1,2,\cdots,n\}$ for simplicity.
The family
\[
\left\{
v_{J,\beta}:
J\subseteq[n],\quad
\beta\in\Delta_{N-|J|-1}^{(n)}
\right\}
\]
is linearly
independent. Therefore, $\mathscr D\mathbf f=0$ if and only if
\[
\sum_{i=1}^n g_J(\beta+e_i)=0
\]
for every $J\subseteq[n]$ and every $\beta\in\Delta_{N-|J|-1}^{(n)}$.
This proves the equivalence of \emph{(2)} and \emph{(3)}, and hence all
three statements are equivalent.
\end{proof}

\begin{rmk}\label{rmk:residue-up-to-Walsh}
	Indeed from the proof above, we have the following identity:
	\begin{equation}\label{eq:reccurence-Df-identity}
		\mathscr{D}\mathbf f=  \sum_{y\in S_{N-1}}\left(\sum_{\substack{x\in S_N\\|x-y|_1=1}}f(x)\right)e_{y_1}\otimes \cdots \otimes e_{y_n}= \sum_{\substack{J\subseteq[n]\\|J|\leq N-1}}\ \sum_{\beta\in\Delta_{n-|J|-1}^{(n)}}
\left(
\sum_{i=1}^n g_J(\beta+e_i)
\right)v_{J,\beta}.
	\end{equation}
This precisely means that the following two residue functions,
\[{\rm Res}_f: \mathbb S_{N-1}\rightarrow \R,\quad {\rm Res}_f(y):= \sum_{\substack{x\in S_N\\|x-y|_1=1}}f(x) \]
and 
\[{\rm Res}_g(J,\beta):= \sum_{i=1}^n g_J(\beta+e_i), \quad J\subset\{1,\cdots,n\}, \beta\in \Delta^{(n)}_{N-1-|J|}, \]
are related by a Walsh transform.
\end{rmk}

From above argument, it's natural to believe that there must be some connection between QUC on $\Z^n$ and QUC on simplex lattice. Indeed, the following lemma is the convertion of Theorem \ref{thm:robust-anisotropic} from simplex lattice to the $\ell^1$-sphere in $\Z^n$.

We apply this decomposition to a first non-small sphere. 
Recalling \eqref{eq:anisotropic-weight}, for $a\in\Z_{\ge 0}^n$, let
$A_1\ge\cdots\ge A_n$ be the decreasing rearrangement of its
coordinates and we define
\begin{equation}\label{eq:Phi-lattice}
 \phi_n(a)=\Phi_n(A_1,A_2,\cdots,A_n)=
 \frac{(A_n +1)\prod_{i=1}^{n-2}(A_i+1)}{(1+|a|)^{h}},
 \qquad h=\left\lfloor\frac{n-2}{2}\right\rfloor.
\end{equation}

\begin{lem}[Transversality on the $\ell^1$-sphere]
\label{lem:first-sphere}
Let $n \ge3$, $K\ge0$. There are constant $C_{n,K},c_n>0$ such that the following things holds: Let $z\in\Z^n$ and $r\ge1$.  Assume that
$u,V$ satisfy $\Delta_{\Z^n} u=Vu$ in the $\ell^1$-ball
$\{y:|y-z|_1\le r-1\}$, with $\|V\|_\infty\le K$.  Suppose that there is a $x\in\mathbb S_r(z)$ satisfies $|u(x)|=A>0$
and 
\begin{equation}\label{eq:small-inner-ball}
 |u(y)| \leq \exp(-C_{n,K}r)A,\quad \forall  \ |y-z|_1\leq r-1.
\end{equation}
Then
\begin{equation}\label{eq:cardinality-on-sphere}
	\#\left\{ y\in \mathbb S_{r}: |u(y)|\geq \exp(-C_{n,K}r)A  \right\} \geq  c_n \phi_n(|x_1-z_1|,\ldots,|x_n-z_n|).
\end{equation}
\end{lem}

\begin{proof}
Without loss of generality, we can translate $z$ to the origin and assume $z=0$. Let $f=u|_{\mathbb S_r}$.  At every
$y\in \mathbb S_{r-1}$, the equation \eqref{eq:laplacian-convention} gives
\begin{equation}\label{eq:approx-inward-equation}
 \res_f(y)= \sum_{\substack{x'\sim y\\x'\in \mathbb S_r}}f(x')
 =(2n+V(y))u(y)
 -\sum_{\substack{x'\sim y\\x'\in \mathbb S_{r-2}}}u(x').
\end{equation}
By condition \eqref{eq:small-inner-ball}, the right-hand side has modulus at most
\begin{equation}\label{eq:inward-residual-bound}
 \left| \res_{f} (y) \right| \leq (4n+K)\exp(-C_{n,K}r)A,\ \forall y\in \mathbb S_{r-1}.
\end{equation}
By Lemma \ref{lem:mathscrD-equivalence} and Remark \ref{rmk:residue-up-to-Walsh}, if we still decode $f$ by the tensor vector \eqref{eq:sphere-vector}, the Walsh coefficients $g_J(b)$ will be determined and we take the residue fucntion 
\[{\rm Res}_g(J,\beta):= \sum_{i=1}^n g_J(\beta+e_i), \quad J\subset\{1,\cdots,n\}, \beta\in \Delta^{(n)}_{r-1-|J|}. \]
The functions $\res_f$ and $\res_g$ differ only by a Walsh transform. Therefore by Theorem \ref{thm:Walsh-transform}, we have 
	\begin{equation}\label{eq:Walsh-trans-on-residue}
 \res_g(J,y-\one_J)=2^{-n}
 \sum_{\sigma \in \{\pm 1\}^ {I(y)} }\sigma_J \res_f(\sigma y), \ \forall \ y\in \Delta^{(n)}_{r-1}, J\subset I(y).
 \end{equation}
Therefore, by \eqref{eq:inward-residual-bound} and \eqref{eq:Walsh-trans-on-residue}, we have 
\[\sup_{y,J}|\res_g(J,y-\one_J)| \leq 2^{-n+|I(y)|} (4n+K)\exp(-C_{n,K}r)A \leq (4n+K)\exp(-C_{n,K}r)A.\]
This is equivalent to say 
\begin{equation}\label{eq:inward-residual-bound-Walsh-coefficient}
 \left| \sum_{i=1}^n g_J(\beta+e_i)\right| \leq (4n+K)\exp(-C_{n,K}r)A,\ \forall \ J\subset \{1,\cdots,n\},\beta\in \Delta^{(n)}_{r-1-|J|}.
\end{equation}
Moreover, since $f(y)$ and $g_J(b)$ aslo differ by a Walsh transform, Theorem \ref{thm:Walsh-transform} tells us that 
\begin{equation}\label{eq:revers-Walsh-trans-in-transversality}
 f(x)=2^{n-|I(x_+)|}
 \sum_{J\subset I(x_+)}\sigma_J g_J(x_+-\one_J).
\end{equation}
Here we denote $x_+=(|x_1|,|x_2|,\cdots,|x_n|)$ and $\sigma$ be the unique sign sequence such that $x=\sigma x_+$. By condition $|f(x)|=A>0$, we have 
\[
   A\leq 2^{n-|I(x_+)|}\sum_{J\subset I(x_+)}|\sigma_J|\cdot \sup_{J}|g_J(x_+-J)|= 2^n \sup_{J}|g_J(x_+-J)|.
\]
This ensures that we can find a $J_0\subset I(x_+)$ such that 
\begin{equation}\label{eq:transversality-point-of-gJ}
	|g_{J_0}(x_+-\one_{J_0})|\geq 2^{-n}A.
\end{equation}
Take $J=J_0$ in \eqref{eq:inward-residual-bound-Walsh-coefficient} yields
\begin{equation}
	\left| \sum_{i=1}^n g_{J_0}(\beta+e_i)\right| \leq 2^n (4n+K)\exp(-C_{n,K}r)|g_{J_0}(x_+-\one_{J_0})|.
\end{equation}

Denote $b=x_+-\one_{J_0}\in \Delta^{(n)}_{r-|J_0|}$. Now let $C'_n,c'_n$ be the feasible constant in Theorem \ref{thm:robust-anisotropic}. By taking $C_{n,K}$ sufficiently large, we have 
   \[2^n (4n+K)\exp(-C_{n,K}r) < \exp(-C'_n r)\leq \exp(-C'_n (r-|J_0|)).\]
and therefore $g_{J_0}$ satisfies condition \eqref{eq:robust-simplex-residual}.

Assume first that $r-|J_0|\geq 1$. The application of Theorem \ref{thm:robust-anisotropic} to $g_{J_0}$ on $\Delta^{(n)}_{r-|J_0|}$ with 
\begin{equation}\label{eq:sector-anchor-shell}
 |g_{J_0}(b)|\ge2^{-n}A
\end{equation}
will yield
\begin{equation}
 \#\left\{\alpha\in\Delta_{r-|J_0|}^{(n)}:
 |g_{J_0}(\alpha)|\geq \econst^{-C'_n(r-|J_0|)}|g_{J_0}(b)|\geq 2^{-n}\econst^{-C'_n r} A \right\}
 \geq c'_n\Phi_n(\mathbf B),
\end{equation}
Here $\mathbf B=(B_1,\cdots,B_n)$ is rearrangement of $b$ in \eqref{eq:ordered-a} order, and
\begin{equation}
 \qquad
 \Phi_n(\mathbf B):=
 \frac{(B_n+1)\prod_{i=1}^{n-2}(B_i+1)}{(r-|J_0|+1)^{h}},
 \qquad
 h=\left\lfloor\frac{n-2}{2}\right\rfloor.
\end{equation}
Since $b=x_+-\one_{J_0}$. Denote $a=x_+$, then each coordinate of $b$ is either $a_i$ or $a_i-1$, and the second case
occurs only when $a_i\ge1$.  Hence
\begin{equation}\label{eq:order-statistic-stability}
 b_i+1\ge\frac12(a_i+1),\ \forall 1\leq i\leq n.
\end{equation}
Rearrangement will not distroy \eqref{eq:order-statistic-stability} and respectively
\begin{equation}\label{eq:order-statistic-stability-rearrange}
 B_i+1\ge\frac12(A_i+1),\ \forall 1\leq i\leq n.
\end{equation}
Therefore,
\begin{align}\label{eq:Phi-sector-comparison}
 \Phi_n(B)\ge 2^{-(n-1)}\frac{(A_n+1)\prod_{i=1}^{n-2}(A_i+1)}{(r-|J_0|+1)^{h}} \geq 2^{-(n-1)}\phi_n(a),
\end{align}
and we obtain
\begin{equation}\label{eq:QUC-gJ0}
 \#\left\{\alpha\in\Delta_{r-|J_0|}^{(n)}:
 |g_{J_0}(\alpha)|\geq 2^{-n}\econst^{-C'_n r} A \right\}
 \geq c'_n 2^{-(n-1)}\phi_n(a).
\end{equation}

If $r-|J_0|=0$, then $J=I(a)$ and every nonzero coordinate of $a$ equals
$1$.  Hence $\phi_n(a)\lesssim_n 1$, while by \eqref{eq:transversality-point-of-gJ} the point $b$ lies in 
the set of left hand side of \eqref{eq:QUC-gJ0} (therefore the cardinality is larger than $1$). Hence, we still have 
\begin{equation}\label{eq:QUC-gJ0-extrm-case}
 \#\left\{\alpha\in\Delta_{r-|J_0|}^{(n)}:
 |g_{J_0}(\alpha)|\geq 2^{-n}\econst^{-C'_n r} A \right\}
 \gtrsim_n \phi_n(a)
\end{equation}
when $r-|J_0|=0$.

By choosing $c_n$ sufficiently small and $C_{n,K}$ sufficiently large, \eqref{eq:QUC-gJ0} and \eqref{eq:QUC-gJ0-extrm-case} conclude that 
\begin{equation}\label{eq:QUC-gJ0-final}
 \#\left\{\alpha\in\Delta_{r-|J_0|}^{(n)}:
 |g_{J_0}(\alpha)|\geq \econst^{-C_{n,K} r} A \right\}
 \geq c_n \phi_n(a).
\end{equation}

Finally, we recover the transversality of $f$ from the transversality of the fiber $g_{J_0}$. By Walsh transform, for each $\alpha$ lies in the set of left hand side of \eqref{eq:QUC-gJ0-final}, we have 
\begin{equation}\label{eq:Walsh-trans-J_0}
 g_{J_0}(\alpha)=2^{-n}
 \sum_{\sigma \in \{\pm 1\}^ {I(\alpha+\one_{J_0})} }\sigma_{J_0} f(\sigma (\alpha+\one_{J_0})).
\end{equation}
Therefore, again we can find a $\sigma^{(\alpha)}\in \{\pm 1\}^ {I(\alpha+\one_{J_0})}$ such that 
\[|f(\sigma^{(\alpha)} (\alpha+\one_{J_0}))|\geq |g_{J_0}(\alpha)|\geq \econst^{-C_{n,K} r} A. \]
The $\sigma^{(\alpha)} (\alpha+\one_{J_0})\in \mathbb S_r$ are distinct for different $\alpha$, because that the absolute value index of $\sigma^{(\alpha)} (\alpha+\one_{J_0})$ is $\alpha+\one_{J_0}\in \Delta^{(n)}_r$, which varies about different $\alpha$. The above discussion implies that 
\begin{equation}\label{eq:QUC-gJ0-final}
 \#\left\{ y \in \mathbb S_r:
 |f(y)|\geq \econst^{-C_{n,K} r} A \right\}
 \geq c_n \phi_n(a).
\end{equation}
This proves \eqref{eq:cardinality-on-sphere} since $a=x_+=(|x_1|,\cdots,|x_n|)$.

\end{proof}

\subsection{The Voronoi estimate}\label{sec:voronoi}

In this section, we do the so-called Voronoi estimate (similar technique is also used in 
\cite{Aurenhammer1991,OkabeEtAl2000,Erwig2000,LeeWong1980}). We first investigate the lower bound function $\phi_n(a)$ in Lemma \ref{lem:first-sphere}.

\begin{lem}
\label{lem:reciprocal-sum}
For $n\ge 3$, $h=\lfloor(n-2)/2\rfloor$, and $R \ge2$,
\begin{equation}\label{eq:reciprocal-sum}
 \sum_{a\in\{0,1,\ldots,R\}^n}\frac1{\phi_n (a)}
 \lesssim_n R^{h+1}\log(2+R)
 =R^{n-q}\log(2+R).
\end{equation}
Here $q=\lceil n/2\rceil$. Moreover, if one extends $\phi_n(a)$ from $\Z^n_{\geq 0}$ to the whole $\Z^n$ by 
\[\phi_n(a)=\phi_n(|a_1|,\cdots,|a_n|),\] 
we have the same estimate for summation about $a\in Q_R$.
\end{lem}

\begin{proof}
The sign $\{\pm 1\}^n$ and the coordinate permutations cost at most $2^n n!$. Hence it is enough
to estimate the summation over
\[
 R+1\ge x_1\ge x_2\ge\cdots\ge x_n \ge1,
 \qquad x_i=A_i+1.
\]
Since $1+|a|=1+\sum_{i}A_i\le n x_1$, the summation is bounded by
\begin{equation}\label{eq:ordered-summand}
 \lesssim_n  \sum_{1\leq x_n\leq \cdots\leq x_1\leq R+1} \frac{x_1^h}{x_1x_2\cdots x_{n-2}x_n}.
\end{equation}
We claim that, for every $k\ge2$,
\begin{equation}\label{eq:nested-harmonic-claim}
 \sum_{1\le y_k\le\cdots\le y_1\le M}
 \frac1{y_1\cdots y_{k-2}y_k}
 \lesssim_k
 M\log(2+M).
\end{equation}
For $k=2$, the left-hand side is
\[\sum_{y_1\le M}\sum_{y_2\le y_1}y_2^{-1}\lesssim \sum_{y_1\leq M} \log(y_1+1) \lesssim M\log(2+M).\]
If \eqref{eq:nested-harmonic-claim} holds for $k-1$, then the left side
for $k$ is at most
\[
 \sum_{y_1=1}^M\frac1{y_1}
 \left( \sum_{1\leq y_k\leq \cdots \leq y_2\leq y_1} \frac{1}{y_2\cdots y_{k-2}y_k}\right)
 \lesssim \sum_{y_1=1}^M\frac1{y_1}
 y_1\log(2+y_1)
 \lesssim_k M\log(2+M),
\]
which proves the claim by induction.

Fixing $x_1$ in \eqref{eq:ordered-summand} and applying
\eqref{eq:nested-harmonic-claim} with $k=n-1$ bounds the sum over
$x_2,\ldots,x_n$ by
$\lesssim_n x_1\log(2+x_1)$  Therefore 
\[
 \text{\eqref{eq:ordered-summand}}\lesssim_n \sum_{x_1=1}^{R+1}x_1^h\log(2+x_1)
 \lesssim_n R^{h+1}\log(2+R).
\]
This proves \eqref{eq:reciprocal-sum}.
\end{proof}

Next we do the Voronoi double-counting estimate: 
For a set
$S\subset\Z^n$ and $z\in\Z^n$, we denote
\begin{equation}\label{eq:nearest-set}
 r_S(z)=\dist_1(z,S),
 \qquad
 \mathcal N_S(z)=\{x\in S:|x-z|_1=r_S(z)\},
 \qquad
 m_S(z)=|\mathcal N_S(z)|.
\end{equation}

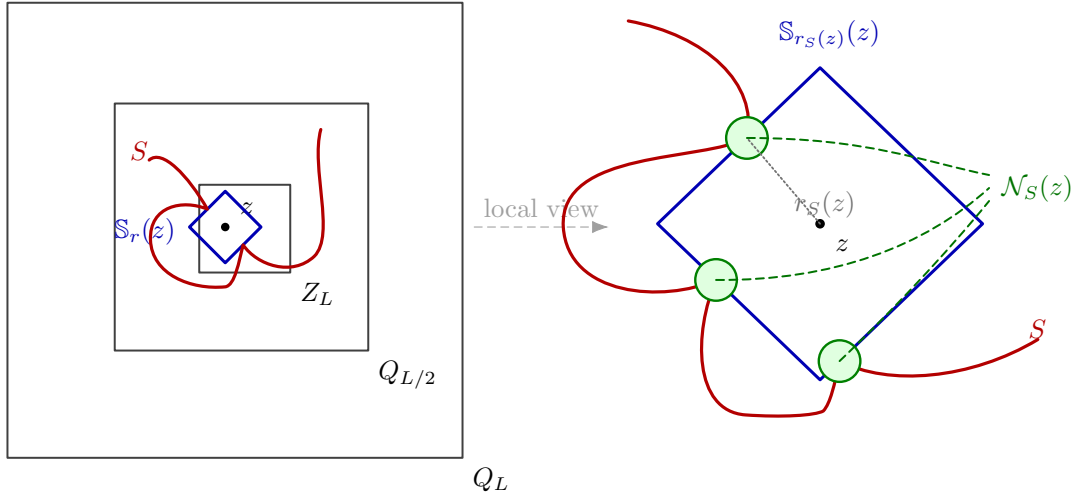
\begin{figure}[ht]
\centering
\begin{tikzpicture}[
  x=0.86cm,
  y=0.86cm,
  line cap=round,
  line join=round,
  box/.style={
    draw=black!75,
    line width=0.75pt
  },
  sphere/.style={
    draw=blue!70!black,
    line width=1.15pt
  },
  set/.style={
    draw=red!70!black,
    line width=1.20pt
  },
  nearest/.style={
    circle,
    draw=green!50!black,
    fill=green!12,
    line width=0.9pt,
    minimum size=5.5mm,
    inner sep=0pt
  },
  guide/.style={
    draw=green!45!black,
    densely dashed,
    line width=0.75pt
  },
  every node/.style={font=\small}
]


\draw[box] (0,0) rectangle (7,7);
\node[anchor=north west] at (7,0) {$Q_L$};

\draw[box] (1.65,1.65) rectangle (5.55,5.45);
\node[anchor=north west] at (5.55,1.65) {$Q_{L/2}$};

\draw[box] (2.95,2.85) rectangle (4.35,4.20);
\node[anchor=north west] at (4.35,2.85) {$Z_L$};

\coordinate (zsmall) at (3.35,3.55);
\fill (zsmall) circle (1.6pt);
\node[anchor=south west] at (3.42,3.59) {$z$};

\draw[sphere]
  (3.35,4.10) --
  (3.90,3.55) --
  (3.35,3.00) --
  (2.80,3.55) -- cycle;

\node[
  blue!70!black,
  anchor=east
] at (2.74,3.48) {$\mathbb S_r(z)$};

\draw[set]
  (2.18,4.58)
  .. controls (2.38,4.76) and (2.78,4.34) ..
  (3.075,3.825)
  .. controls (2.70,3.98) and (2.18,3.88) ..
  (2.20,3.35)
  .. controls (2.22,2.82) and (2.82,2.58) ..
  (3.35,2.63)
  .. controls (3.50,2.65) and (3.56,3.02) ..
  (3.625,3.275)
  .. controls (3.84,2.96) and (4.38,2.78) ..
  (4.72,3.10)
  .. controls (5.05,3.43) and (4.66,4.47) ..
  (4.82,5.05);

\node[
  red!70!black,
  anchor=east
] at (2.32,4.72) {$S$};

\draw[
  gray!75,
  densely dashed,
  -{Latex[length=2.2mm]}
]
  (7.18,3.55) -- (9.25,3.55)
  node[midway,above,gray!75] {local view};


\coordinate (zlarge) at (12.50,3.60);
\coordinate (top)    at (12.50,6.00);
\coordinate (right)  at (15.00,3.60);
\coordinate (bottom) at (12.50,1.20);
\coordinate (left)   at (10.00,3.60);

\draw[sphere]
  (top) -- (right) -- (bottom) -- (left) -- cycle;

\node[
  blue!70!black,
  anchor=south
] at (12.65,6.13)
  {$\mathbb S_{r_S(z)}(z)$};

\fill (zlarge) circle (1.8pt);
\node[anchor=north west] at (12.58,3.52) {$z$};

\coordinate (a) at (11.375,4.920);
\coordinate (b) at (10.900,2.736);
\coordinate (c) at (12.800,1.488);

\draw[set]
  (9.55,6.72)
  .. controls (10.75,6.50) and (11.55,5.75) ..
  (a)
  .. controls (10.70,4.55) and (8.48,4.75) ..
  (8.55,3.55)
  .. controls (8.62,2.48) and (10.10,2.38) ..
  (b)
  .. controls (10.55,1.92) and (10.55,0.72) ..
  (11.35,0.66)
  .. controls (11.92,0.62) and (12.42,0.65) ..
  (12.55,0.72)
  .. controls (12.66,0.84) and (12.72,1.12) ..
  (c)
  .. controls (13.55,1.10) and (14.82,1.20) ..
  (15.85,1.82);

\node[
  red!70!black,
  anchor=west
] at (15.55,1.98) {$S$};

\node[nearest] at (a) {};
\node[nearest] at (b) {};
\node[nearest] at (c) {};

\draw[
  black!55,
  densely dotted,
  line width=0.7pt
]
  (zlarge) -- (a)
  node[midway,below right] {$r_S(z)$};

\draw[guide]
  (a)
  .. controls (13.25,4.92) and (14.10,4.55) ..
  (15.10,4.35);

\draw[guide]
  (b)
  .. controls (12.70,2.72) and (14.10,3.25) ..
  (15.10,4.15);

\draw[guide]
  (c)
  .. controls (13.55,2.20) and (14.45,3.15) ..
  (15.10,3.95);

\node[
  green!45!black,
  anchor=west
] at (15.14,4.15)
  {$\mathcal N_S(z)$};

\end{tikzpicture}

\caption{
A two-dimensional schematic of \eqref{eq:nearest-set} and Lemma \ref{lem:weighted-voronoi}.
}
\label{fig:voronoi-nearest-set}
\end{figure}

\begin{lem}[Voronoi estimate]
\label{lem:weighted-voronoi}
Let $S\subset Q_{L/2}$ contain the origin, and set
\[
 Z_L=Q_{\lfloor L/(8n)\rfloor}.
\]
Suppose that for every $z\in Z_L$ and every $x\in\mathcal N_S(z)$, we have 
\begin{equation}\label{eq:nearest-multiplicity-hypothesis}
 m_S(z)\gtrsim_n \phi_n(|x_1-z_1|,\ldots,|x_d-z_d|).
\end{equation}
Then
\begin{equation}\label{eq:weighted-voronoi-conclusion}
 |S|\gtrsim_n \frac{L^q}{\log(2+L)}.
\end{equation}
Here $q=\lceil n/2\rceil $.
\end{lem}

\begin{proof}
For fixed $z$, \eqref{eq:nearest-multiplicity-hypothesis} gives
\[
 \sum_{x\in\mathcal N_S(z)}
 \frac1{\phi_n(|x_1-z_1|,\ldots,|x_d-z_d|)}
 \gtrsim_n \frac{m_S(z)}{m_S(z)}=1.
\]
Sum over $z\in Z_L$ and interchange the two finite sums:
\begin{align}
 |Z_L|
 &\lesssim_n
 \sum_{x\in S}\sum_{\substack{z\in Z_L\\x\in\mathcal N_S(z)}}
 \frac1{\phi_n(|x_1-z_1|,\ldots,|x_n-z_n|)}\notag\\
 &\le
 \sum_{x\in S}\sum_{z\in Z_L}
 \frac1{\phi_n(|x_1-z_1|,\ldots,|x_n-z_n|)}\notag\\
 &\lesssim_n |S|L^{n-q}\log(2+L),
 \label{eq:voronoi-double-count}
\end{align}
where Lemma~\ref{lem:reciprocal-sum} was used in the last line.
Since $|Z_L|\sim_n L^n$, rearranging
\eqref{eq:voronoi-double-count} proves
\eqref{eq:weighted-voronoi-conclusion}.
\end{proof}

\subsection{Geometric pigeonholing argument}

Finally, we use a geometric pigeonholing argument to obtain the structure needed in Lemma \ref{lem:first-sphere}, and then prove Theorem \ref{thm:main}. 
Such argument also plays a key role in \cite{LZ22}.

\begin{proof}[Proof of Theorem~\ref{thm:main}]
Let
\begin{equation}\label{eq:M-definition}
 M=6dL.
\end{equation}
Fix a constant $B=B(n,K)$ to be chosen sufficiently large, and put
\begin{equation}\label{eq:amplitude-levels}
 U_j=\exp(-BL\cdot j)|u(0)|,
 \qquad
 S_j=\{x\in Q_{L/2}:|u(x)|\ge U_j\},
 \qquad 0\le j\le M.
\end{equation}
Then we have $S_0\subset S_1\subset S_2 \subset \cdots$ and $0\in S_0$. We will keep the notation in \eqref{eq:nearest-set}. Let 
\begin{equation}\label{eq:nearest-set-ordered}
 r_j(z)=\dist_1(z,S_j),
 \quad
 \mathcal N_j(z)=\{x\in S_j:|x-z|_1=r_j(z)\},
 \quad
 m_j(z)=|\mathcal N_j(z)|.
\end{equation}
Clearly, for arbitrary fixed $z\in Z_L=Q_{\lfloor L/(8n)\rfloor}$, we have 
\begin{equation}\label{eq:monotone-radii}
  0\leq r_{j+1}(z)\leq r_{j}(z) \leq |z|_1\leq L/8,\quad 1\leq m_j(z)\leq (L+1)^d {\rm \ for \ all \ } 0\leq j\leq M-1
\end{equation}
Therefore, since $r_j(z)\in \Z$, there are at most $L/8$ many strict decreasing (i.e. $r_j(z)>r_{j+1}(z)$) in $r_j(z)$.

Now we study the case when $r_j(z)=r_{j+1}(z)=r\geq 1$ for an index $j$. In this case, we have that the $\ell^1$-sphere
\[\mathbb{S}_{r-1}(z)\cap S_{j+1}=\varnothing. \]
Hence, for each $x\in \mcN_j(z)$, we have 
\begin{equation}\label{eq:two-sphere-quite}
  |u(y)|< U_{j+1} \leq \exp(-BL)U_{j}\leq \exp(-BL)|u(x)| \quad {\rm  for\ all \ } |y-z|_1<1.
\end{equation}
Assume $C'_{n,K},c'_n$ are the feasible constant in Lemma \ref{lem:first-sphere}. Choose $B\gg C'_{n,K}$ sufficiently large to ensure that 
\[\exp(-BL)\leq \exp(-8Br)\leq \exp(-C'_{n,K}r).\]
Then \eqref{eq:two-sphere-quite} exactly matches with the condition \eqref{eq:small-inner-ball}. Therefore, by applying Lemma \ref{lem:first-sphere} to $\mathbb S_r(z)$, we prove that 
\begin{equation}\label{eq:cardinality-on-pigeonholing-sphere}
	\#\left\{ y\in \mathbb S_{r}(z): |u(y)|\geq \exp(-C'_{n,K}r)|u(x)| \right\} \geq  c'_n \phi_n(|x_1-z_1|,\ldots,|x_n-z_n|).
\end{equation}
The points in the set of left hand side of \eqref{eq:cardinality-on-pigeonholing-sphere} will satisfy
\[
 |u(y)|\ge\exp(-C'_{n,K}r)|u(x)|
 \ge\exp(-BL)U_j=U_{j+1}.
\]
So every such point belongs to $S_{j+1}$, and its distance from $z$ is
$r$.  It therefore belongs to $\mathcal N_{j+1}(z)$. This implies that 
\begin{equation}\label{eq:cardinality-for-stable-r}
  m_{j+1}(z)=\# \mcN_{j+1}(z)\geq  c'_n \phi_n(|x_1-z_1|,\ldots,|x_n-z_n|)
\end{equation}
for all $x\in \mcN_j(z)$.

Define the reciprocal weight by 
\begin{equation}
  \mcW_r:=\sum_{x:x\in S_J \atop |x-z|_1=r }\frac{1}{\phi_n(|x_1-z_2|,\ldots,|x_n-z_n|)}, \qquad \mcW:=\sum_{r\geq 0}\mcW_r.
\end{equation}
We claim that $\mcW\geq c'_n/4$. Suppose instead, $\mcW<c'_n/4$. We must have $z\notin S_J$, since otherwise 
\[\mcW  \geq \mcW_0=\phi_n(0)^{-1}= 1>c'_n/4.\]
Therefore $r_j(z)\geq 1$ for all $0\leq j\leq J$. Now for any $r\geq 1$, let us define 
\begin{equation}\label{eq:stable-trans-times}
  k_r:=\#\left\{  0\leq j\leq J-1 :   r_j(z)=r_{j+1}(z)=r   \right\}.
\end{equation}
For index $j$ in the set of right hand side of \eqref{eq:stable-trans-times}, we have \eqref{eq:cardinality-for-stable-r} holds, and therefore
\begin{equation}\label{eq:stable-ratio}
  \mcW_r\geq \sum_{x\in \mcN_j(z)}\frac{c'_n}{m_{j+1}(z)}=c'_n\frac{m_j(z)}{m_{j+1}(z)}.
\end{equation}
Therefore, applying \eqref{eq:stable-ratio} for $k_r$ times and producting them together yield
\begin{equation}\label{eq:ratio-product}
  \left(\frac{c'_n}{\mcW_r}\right)^{k_r}\leq \frac{m_{j+k_r}(z)}{m_j(z)} \ ({\rm for\ some \ }j)\leq (L+1)^d,
\end{equation}
where for the last inequality we used \eqref{eq:monotone-radii}.

Let 
\[{\rm Rad}_{\leq}:=\left\{r\geq 1: \mcM_r\leq c'_n L^{-1/2}\right\},\]
\[{\rm Rad}_{>}:=\left\{r\geq 1: \mcM_r> c'_n L^{-1/2}\right\}.\]
For $r\in {\rm Rad}_{\leq}$, \eqref{eq:ratio-product} yields 
\[L^{k_r/2}\leq(1+L)^d \leq L^{2d}\quad \Longrightarrow \quad k_r\leq 4d.\]
For $r\in{\rm Rad}_{>}$, the assumption $\mcW<c'_n/4$ will yield 
\[\#{\rm Rad}_{>}\cdot (c_n' L^{-1/2}) \leq \mcW\leq c'_n/4 \quad \Longrightarrow \quad \#{\rm Rad}_{>}\leq \sqrt{L}/4; \]
and in this case \eqref{eq:ratio-product} will yield 
\[ 4^{k_r} \leq \left(\frac{c'_n}{\mcW}\right)^{k_r}\leq \left(\frac{c'_n}{\mcW_r}\right)^{k_r} \leq L^{2d}\quad \Longrightarrow \quad k_r\leq 2d\log L/\log 4.\]
Therefore, we have 
\begin{align}
  J &\leq \frac{L}{8}+\sum_{r\in {\rm Rad}_{\leq}}k_r +\sum_{r\in {\rm Rad}_{>}}k_r \\
     & \leq \frac{L}{8}+ 4dL+ \frac{d\sqrt{L}\log L}{4\log 2}<6dL=J
\end{align}
when $R\gg_d 1$, a contradiction! Therefore we must have $\mcW\geq c'_n/4$ for any $z\in Z_L$. Finally, sum the quantity
$\mcW$ for all $z\in Z_L$, and by the similar argument as \eqref{eq:voronoi-double-count}, we have 
\[\frac{c'_n}{4}|Z_L|\leq \sum_{x\in S}\sum_{\substack{z\in Z_L\\x\in\mathcal N_S(z)}}
 \frac1{\phi_n(|x_1-z_1|,\ldots,|x_n-z_n|)}\lesssim_n |S_J|L^{n-q}\log(2+L).\]
 Since $|Z_L|\gtrsim_n L^n$, we have 
 \[|S_J|\gtrsim_n \frac{L^q}{\log (2+L)}.\]
 This proves \eqref{eq:main-conclusion} by $U_J\geq \exp(-6dBL^2)$.

\end{proof}

\appendix

\section{Proof of Lemma \ref{lem:factorial-tableau}}
\begin{proof}[Proof of Lemma \ref{lem:factorial-tableau}]

For $z=(z_1,\ldots,z_r)$, $a=(a_1,a_2,\ldots)$ and an integer $m$, we use the generating function
\[
\displaystyle
 \prod_{v=1}^r\frac1{1-tz_v}
 \prod_{h=1}^{r+m-1}(1-ta_h)
\]
to define
\[
 H_m(z\mid a):=
 \begin{cases}
 [t^m]\displaystyle
 \prod_{v=1}^r\frac1{1-tz_v}
 \prod_{h=1}^{r+m-1}(1-ta_h),&m\geq0,\\[3pt]
 0,&m<0.
 \end{cases}
\]
Here, for a function $f(t)$, we denote $[t^m]f$ be the coefficient of $t^m$ in the power series of $f(t)$. We also assume that empty products have value $1$, so $H_0(z\mid a)=1$.  Put
$\tau a=(a_2,a_3,\ldots)$.  For $r,m\geq1$, directed computation gives
\[
 \begin{aligned}
 &(z_1-a_1)H_{m-1}(z_1,\ldots,z_r\mid\tau a)
   +H_m(z_2,\ldots,z_r\mid\tau a)\\
 &=[t^m]\left(\frac{t(z_1-a_1)}{1-tz_1}+1\right)
   \prod_{v=2}^r\frac1{1-tz_v}
   \prod_{h=2}^{r+m-1}(1-ta_h)\\
 &=[t^m]\frac{1-ta_1}{1-tz_1}
   \prod_{v=2}^r\frac1{1-tz_v}
   \prod_{h=2}^{r+m-1}(1-ta_h)
  =H_m(z_1,\ldots,z_r\mid a).
 \end{aligned}
\]
On the other hand, we define another quantity by 
\[
 F_m(z\mid a):=
 \sum_{1\leq v_1\leq\cdots\leq v_m\leq r}
 \prod_{p=1}^m(z_{v_p}-a_{v_p+p-1}),
 \qquad F_0(z\mid a):=1.
\]
We have 
\begin{align*}
	F_m(z\mid a)& = \left(\sum_{1\leq v_1\leq\cdots\leq v_m\leq r \atop v_1=1}+\sum_{1\leq v_1\leq\cdots\leq v_m\leq r \atop v_1>1}\right) \prod_{p=1}^m(z_{v_p}-a_{v_p+p-1}) \\
	            &=(z_1-a_1)\sum_{1\leq v_2\leq\cdots\leq v_m\leq r } \prod_{p=2}^m(z_{v_p}-a_{v_p+p-1}) + \sum_{2\leq v_1\leq\cdots\leq v_m\leq r} \prod_{p=1}^m(z_{v_p}-a_{v_p+p-1})\\ 
	            &=(z_1-a_1)\sum_{1\leq v_1\leq\cdots\leq v_{m-1}\leq r } \prod_{p=1}^{m-1}(z_{v_p}-\tau a_{v_p+p-1}) + \sum_{1\leq v_1\leq\cdots\leq v_m\leq r-1} \prod_{p=1}^m(z_{v_p+1}-\tau a_{v_p+p-1})\\ 
				&=(z_1-a_1)F_{m-1}(z_1,\ldots,z_r\mid\tau a) +F_m(z_2,\ldots,z_r\mid\tau a),
\end{align*}
which exactly matches with the recurrence $H_m(z\mid a)$.
Since $F_0=H_0=1$ and both sides vanish when $r=0<m$,
by induction it's easy to prove
\begin{equation}\label{eq:weight-of-H_m}
	 H_m(z_1,\ldots,z_r\mid a)
 =F_m(z_1,\ldots,z_r\mid a)=\sum_{1\leq v_1\leq\cdots\leq v_m\leq r}
   \prod_{p=1}^m(z_{v_p}-a_{v_p+p-1}).
\end{equation}

In particular, when $r=1$, we have 
\[
 H_m(y_i\mid a)=(y_i\mid a)^m,\quad y_i\in \R
\]
and therefore 
\[\det[(y_i\mid a)^{\lambda_j+q-j}]_{1\leq i,j\leq q}= \det[H_{\lambda_j+q-j}(y_i\mid a)]_{1\leq i,j\leq q}, \]
\[\det[(y_i\mid a)^{q-j}]_{1\leq i,j\leq q}= \det[H_{q-j}(y_i\mid a)]_{1\leq i,j\leq q}. \]

We next prove the row-difference identity.  If $i<k$, then the two
lists $y_i,\ldots,y_{k-1}$ and $y_{i+1},\ldots,y_k$ have the same length,
and the definition of $H_m$ gives, for $m\geq1$,
\[
 \begin{aligned}
 &H_m(y_i,\ldots,y_{k-1}\mid a)
   -H_m(y_{i+1},\ldots,y_k\mid a)\\
 &=[t^m]\left(
   \frac1{\prod_{v=i}^{k-1}(1-ty_v)}
   -\frac1{\prod_{v=i+1}^{k}(1-ty_v)}\right)
   \prod_{h=1}^{m+k-i-1}(1-ta_h)\\
 &=(y_i-y_k)[t^{m-1}]
   \frac{\prod_{h=1}^{m+k-i-1}(1-ta_h)}
        {\prod_{v=i}^{k}(1-ty_v)}\\
 &=(y_i-y_k)H_{m-1}(y_i,\ldots,y_k\mid a).
 \end{aligned}
\]
The identity is valid for $m\leq0$ under the preceding conventions.

Now starting with
$\det[H_{\lambda_j+q-j}(y_i\mid a)]$, perform the following rounds of row
operations.  In round $d=1,\ldots,q-1$, replace row $i$ by row $i$ minus
row $i+1$, in the order $i=1,\ldots,q-d$. This procedure will gives that 
\[
 \det[(y_i\mid a)^{\lambda_j+q-j}]_{1\leq i,j\leq q}
 =\prod_{1\leq i<k\leq q}(y_i-y_k) \cdot
  \det[H_{\lambda_j-j+i}(y_i,\ldots,y_q\mid a)]_{1\leq i,j\leq q}.
\]
If we take $\lambda=(0,\ldots,0)$, the last determinant has entries
$H_{i-j}(y_i,\ldots,y_q\mid a)$; it is lower-triangular matrix with diagonal entries all
$1$.  Consequently
\[
 \det[(y_i\mid a)^{q-j}]_{1\leq i,j\leq q}
 =\prod_{1\leq i<k\leq q}(y_i-y_k),
\]
and the left side of \eqref{eq:factorial-tableau} equals
\[
 \det[H_{\lambda_j-j+i}(y_i,\ldots,y_q\mid a)]_{1\leq i,j\leq q}.
\]

It remains to evaluate this determinant.  We construct $G$ be the finite directed
subgraph of $\Z^2$ with vertices
\[
 (k,\ell)\in \Z^2,\quad
 1\leq k\leq q,
 \qquad q-k+1\leq\ell\leq q+\lambda_1+1.
\]
It contains every south edge $(k,\ell)\to(k+1,\ell)$ whose endpoints are
vertices of $G$, with weight $1$, and every east edge
\[
 (k,\ell-1)\longrightarrow(k,\ell)
 \quad\hbox{whose endpoints are in $G$},
\]
with weight $y_k-a_{k+\ell-q-1}$.  The latter subscript is always
positive, because an east edge in $G$ has
$\ell\geq q-k+2$.  Both coordinates are nondecreasing along every
directed path and at least one increases along every edge, so $G$ is
acyclic.

\begin{figure}[htbp]
\centering
\begin{tikzpicture}[
  x=1.05cm,
  y=1.05cm,
  >=Stealth,
  vertex/.style={circle,fill=black,inner sep=1.3pt},
  edge/.style={gray!55,line width=0.4pt,
    -{Stealth[length=3.2pt]}},
  redpath/.style={red!75!black,line width=1.4pt,
    -{Stealth[length=4.8pt]}},
  bluepath/.style={blue!75!black,line width=1.4pt,
    -{Stealth[length=4.8pt]}},
  greenpath/.style={green!55!black,line width=1.4pt,
    -{Stealth[length=4.8pt]}},
  source/.style={circle,draw=black,fill=white,
    minimum size=5mm,inner sep=0pt},
  sink/.style={rectangle,draw=black,fill=white,
    minimum size=5mm,inner sep=0pt}
]


\foreach \ell in {3,...,7}
  \coordinate (A\ell) at (\ell,2);
\foreach \ell in {2,...,7}
  \coordinate (B\ell) at (\ell,1);
\foreach \ell in {1,...,7}
  \coordinate (C\ell) at (\ell,0);

\foreach \left/\right in {3/4,4/5,5/6,6/7}
  \draw[edge] (A\left)--(A\right);
\foreach \left/\right in {2/3,3/4,4/5,5/6,6/7}
  \draw[edge] (B\left)--(B\right);
\foreach \left/\right in {1/2,2/3,3/4,4/5,5/6,6/7}
  \draw[edge] (C\left)--(C\right);

\foreach \ell in {3,...,7}
  \draw[edge] (A\ell)--(B\ell);
\foreach \ell in {2,...,7}
  \draw[edge] (B\ell)--(C\ell);

\draw[redpath] (A3)--(A4);
\draw[redpath] (A4)--(A5);
\draw[redpath] (A5)--(B5);
\draw[redpath] (B5)--(B6);
\draw[redpath] (B6)--(C6);

\draw[bluepath] (B2)--(B3);
\draw[bluepath] (B3)--(C3);
\draw[bluepath] (C3)--(C4);

\draw[greenpath] (C1)--(C2);

\foreach \p in {A3,A4,A5,A6,A7,
                  B2,B3,B4,B5,B6,B7,
                  C1,C2,C3,C4,C5,C6,C7}
  \node[vertex] at (\p) {};

\node[source] at (A3) {};
\node[anchor=south east] at (A3) {$P_1$};
\node[source] at (B2) {};
\node[anchor=south east] at (B2) {$P_2$};
\node[source] at (C1) {};
\node[anchor=south east] at (C1) {$P_3$};

\node[sink] at (C6) {};
\node[anchor=north] at (C6) {$Q_1$};
\node[sink] at (C4) {};
\node[anchor=north] at (C4) {$Q_2$};
\node[sink] at (C2) {};
\node[anchor=north] at (C2) {$Q_3$};

\draw[-{Stealth[length=4.5pt]},thick]
  (0.8,2.65)--(7.55,2.65)
  node[right] {horizontal coordinate $\ell$};
\draw[-{Stealth[length=4.5pt]},thick]
  (0.35,2.35)--(0.35,-0.45)
  node[below,align=center] {level $k$\\increases downward};
\node[anchor=east] at (0.78,2) {$k=1$};
\node[anchor=east] at (0.78,1) {$k=2$};
\node[anchor=east] at (0.78,0) {$k=3$};

\node[font=\small,fill=white,inner sep=1pt]
  at (6.5,1.25)
  {east weight $y_k-a_{k+\ell-q-1}$};
\node[font=\small,fill=white,inner sep=1pt]
  at (7.32,0.5)
  {south weight $1$};

\end{tikzpicture}
\caption{A schematic picture of the directed graph $G$ for
$q=3$ and $\lambda=(3,2,1)$. Circles denote the starting points $P_i$,
squares denote the terminations $Q_i$, and the colored edges show one
identity-matched vertex-disjoint path family.}
\label{fig:factorial-tableau-graph}
\end{figure}
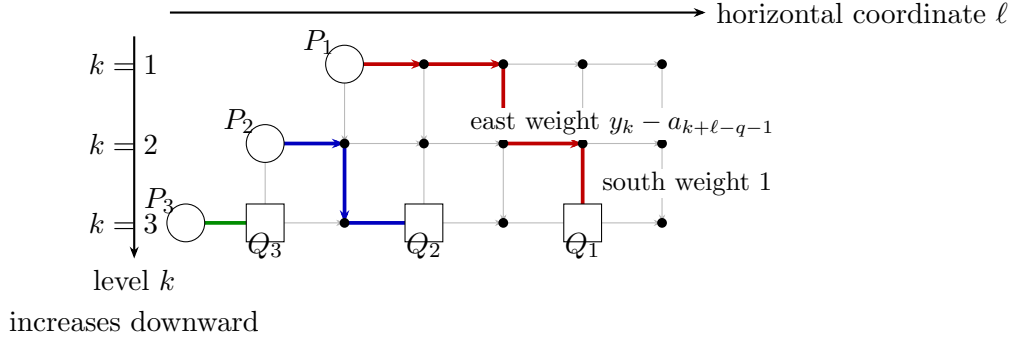

Set
\[
 P_i:=(i,q-i+1),
 \qquad
 Q_j:=(q,q-j+1+\lambda_j).
\]
A path from $P_i$ to $Q_j$ has
$m=\lambda_j-j+i$ east edges; there is no such path when $m<0$.  If the
successive east edges have levels (i.e. the $k$-coordinate)
$i\leq k_1\leq\cdots\leq k_m\leq q$, their weights are
\[
 y_{k_p}-a_{k_p-i+p}\qquad(1\leq p\leq m).
\]
Writing $v_p=k_p-i+1$ and using the formula \eqref{eq:weight-of-H_m} for $H_m$ shows that
the total weight of all paths from $P_i$ to $Q_j$ is precisely
$H_{\lambda_j-j+i}(y_i,\ldots,y_q\mid a)$. Therefore, the matrix $[H_{\lambda_j-j+i}(y_i,\ldots,y_q\mid a)]_{i,j=1}^q$ is exactly the 
edge-weight transport matrix on G from starting points $\{P_1,\cdots, P_q\}$ to terminations $\{Q_1,\cdots, Q_q\}$.

By the Lindstr\"om--Gessel--Viennot Theorem, Theorem \ref{thm:LGV},
\begin{equation*}
  \det[H_{\lambda_j-j+i}(y_i,\ldots,y_q\mid a)]_{1\leq i,j\leq q}
 =\sum_{\sigma\in\mathfrak S_q}\operatorname{sgn}(\sigma)
 \sum_{\substack{(Path_1,\ldots,Path_q)\text{ vertex-disjoint}\\
 Path_i:P_i\to Q_{\sigma(i)}}}\prod_i\operatorname{wt}(Path_i).
\end{equation*}
We claim that in the above summation only the identity matching $\sigma=id$ can occur.  Indeed,
if $i<k$, then when the path from $P_i$ first reaches level $k$, its
column is at least $q-i+1>q-k+1$, the column of $P_k$.  Two
vertex-disjoint south-east paths cannot reverse this strict
left-to-right order, so the first path must end strictly to the right of
the second.  Moreover,
\[
 \operatorname{col}(Q_j)-\operatorname{col}(Q_{j+1})
 =1+\lambda_j-\lambda_{j+1}>0.
\]
Hence a terminal permutation $\sigma$ must satisfy
$\sigma(i)<\sigma(k)$ whenever $i<k$.  It is increasing and therefore
$\sigma=\mathrm{id}$ and 
\begin{equation*}
  \det[H_{\lambda_j-j+i}(y_i,\ldots,y_q\mid a)]_{1\leq i,j\leq q}
 =
 \sum_{\substack{(Path_1,\ldots,Path_q)\text{ vertex-disjoint}\\
 Path_i:P_i\to Q_{i}}}\prod_i\operatorname{wt}(Path_i).
\end{equation*}

For an identity-matched family, let $U(i,j)$ be the level of the $j$-th
east edge of the path from $P_i$ to $Q_i$.  Obviously, we have 
\[U(i,j)\leq U(i,j+1), \ 1\leq j\leq \lambda_i-1\].
Moreover, one may easily check that the condition ``$(Path_1,\ldots,Path_q)\text{ vertex-disjoint}$'' is equavalent to 
\[
 U(i,j)<U(i+1,j)
 \quad\text{for every }(i,j),(i+1,j)\in[\lambda].
\]
That means that there is a bijection between all admissible $(Path_1,\ldots,Path_q)$ and $U\in\operatorname{SSYT}_q(\lambda)$. 
Under this bijection, the $j$-th east edge of path $i$ is from $(U(i,j),q-i+j)$ to $(U(i,j),q-i+1+j)$. Thus it has weight
\[
 y_{U(i,j)}-a_{U(i,j)+j-i}.
\]
Therefore, we have 
\begin{equation*}
  \det[H_{\lambda_j-j+i}(y_i,\ldots,y_q\mid a)]_{1\leq i,j\leq q}
 =
 \sum_{U\in\operatorname{SSYT}_q(\lambda)}\prod_{1\leq i\leq q}\prod_{1\leq j \leq \lambda_i} (y_{U(i,j)}-a_{U(i,j)+j-i}).
\end{equation*}
This proves \eqref{eq:factorial-tableau}.

\end{proof}

\section*{Acknowledgement}
This work  is  supported by  the National Key R\&D Program
of China under Grant 2023YFA1008801.
 Y. Shi is supported by the NSFC(12522110) and  Z. Zhang is  supported by the NSFC (12288101).   
The authors acknowledge the involvement
of LLM during this research. All mathematical arguments and proofs in the final manuscript were
written by the authors. The authors sincerely thank Linjun Li and Linfu Zhang for generously engaging in discussions with us and for the inspiration and suggestions they provided on this work, especially inspiring us to 
refine the previous lower bound $\exp\left(-C_{n,k}\frac{L^{\lceil n/2\rceil+1}}{\log (2+L)}\right)$ to $\exp\left(-C_{n,k}L^2 \right)$ in \eqref{eq:main-conclusion}.

\section*{Data Availability}
		The manuscript has no associated data.
\section*{Declarations}
		{\bf Conflicts of interest} \ The authors  state  that there is no conflict of interest.

\end{document}